\documentclass[12pt]{dalthesis}

\input{rex-cgt-thesis.sty}

\begin{document}

\title{The combinatorial game theory \\ of Reverse Hex}
\author{Jeremiah Hockaday}

\mcs

\degree{Master of Science}
\degreeinitial{M.Sc.}
\faculty{Science}
\dept{Mathematics and Statistics}

\defencemonth{December}\defenceyear{2025}

\nolistoftables

\frontmatter
\begin{abstract}
Rex, short for Reverse Hex, is a set coloring game in which players try to avoid connecting terminals of their color. Combinatorial game theory (CGT) is the study of perfect strategy games. Until recently, both Rex and Hex were not examined through the lens of CGT. In this thesis we take inspiration from the study of normal play games by Berlekamp, Conway, and Guy, along with the combinatorial game theory of Hex developed by Selinger, to develop methods for analyzing Rex positions. We explore how to tell if one position is preferable to another, how to simplify positions, and some special properties of Rex (and antimonotone set coloring games in general). By the end of this thesis we will be able to take a position in a game of Rex, break it into smaller positions, analyze each of the smaller positions, then add the results back together to more easily determine who wins and loses the larger position.
\end{abstract}

\prefacesection{List of Abbreviations and Symbols Used}

CGT $\hfill$ Combinatorial game theory.

Rex $\hfill$ Reverse Hex.

$\leq_A$ $\hfill$ The order given a poset $A$ (page \pageref{def:posets}).

$\inteq$ $\hfill$ Preorder equivalence or the intrinsic equivalence (see pages \pageref{def:posets}, \pageref{intrinsic_order}).

$\leq_o$, $\inteq_o$ $\hfill$ Order based on outcome classes (see page \pageref{concrete_*_antimonotonicity}).

$\leqc$, $\contri$, $\conteq$ $\hfill$ The contextual order (see page \pageref{cont_order}).

$\cateq$ $\hfill$ Isomorphism of posets (see page \pageref{def:monofun}).

$2^X$ $\hfill$ The poset formed by the subsets of $X$ (see page \pageref{subsec:common_posets}).

$\bool$ $\hfill$ The boolean poset (see page \pageref{subsec:common_posets}).

$\deadset$ $\hfill$ The one element poset (see page \pageref{subsec:common_posets}).

$\bot$ $\hfill$ The least element of a poset (see page \pageref{subsec:common_posets}).

$\top$ $\hfill$ The greatest element of a poset (see page \pageref{subsec:common_posets}).

$\hom(A,B)$ $\hfill$ Morphisms from $A$ to $B$ (see page \pageref{subsec:common_posets}).

$\fun{A}$ $\hfill$ Shorthand for $\hom(A, \bool)$ (see page \pageref{subsec:common_posets}).

$\op{a}$ $\hfill$ Elements of the dual of a poset (see page \pageref{subsec:common_posets}).

$\op{f}$ $\hfill$ The dual of a function (see page \pageref{subsec:common_posets}).

$\op{G}$ $\hfill$ The dual of a game (see pages \pageref{def:dual}, \pageref{def:dual2}).

$\lambda$, $\rho$, \emph{and}, \emph{or} $\hfill$ Special functions between posets (see page \pageref{Background:mono_fun}).

$G$, $G^L$, $G^R$, $\gc{G^L}{G^R}$ $\hfill$ Typical game form (see pages \pageref{sec:CGT}, \pageref{def:game_form}).

$G+H$ $\hfill$ Sum of games (see pages \pageref{sec:CGT}, \pageref{sum_on_game}).

$o_L(G)$, $o_R(G)$, $o(G)$ $\hfill$ (Left, right) outcome class of $G$ (see page \pageref{def:LR_outcome_class}).

$\gl$ , $\gn$ , $\gp$ , $\gr$ $\hfill$ Outcome classes (see page \pageref{def:outcome_class}).

$\blackstone*{a}$ $\hfill$ A black stone labeled $a$.

$\whitestone*{a}$ $\hfill$ A white stone labeled $a$.

$\inlinehex*{a}$ $\hfill$ An empty cell labeled $a$.

$*$ $\hfill$ A dead cell (see page \pageref{dead_cell}) and its game form (see page \pageref{SAM_game}).

$[a]$ $\hfill$ Atomic game (see page \pageref{def:game_form}).

$f(G)$ $\hfill$ Image of a game via a monotone function (see page \pageref{map_on_game}).

$G+_f H$ $\hfill$ Shorthand for $f(G+H)$ (see page \pageref{map_on_game}).

$\lcomp{G}{H}$ $\hfill$ The comparison game between $G$ and $H$ (see page \pageref{prop:copycat}).


\begin{acknowledgements}
  I would like to thank my lovely supervisors Svenja Huntemann and Peter Selinger, without whom I could not have written this thesis. I would also like to thank all of my friends at Dalhousie, without whom I could not have \emph{stopped} to take a break. Most importantly, I give thanks to my wonderful mother Angela, who inspired my love of learning at a young age and has always supported me in my endeavors. 
\end{acknowledgements}

\mainmatter

\chapter{Introduction}

Combinatorial Game Theory (CGT) is the study of two player perfect information games with no randomness.
The goal of CGT research is to simplify the process of determining who is winning and losing at any given point of a game.
It has been used for the study of \emph{normal play games} (games where the winner is the person who moves last) to great success.
Over the course of some games the board divides into disjoint regions.
When this happens, it is useful to analyze these disjoint regions separately and ``add'' them back together.

The basic idea of CGT is to find a way to add games together, to define a notion of equivalence on games, and to develop methods for simplifying games.
We assign values to these games, called the \emph{canonical form} of a game, which helps us reason about them more easily.
When we can split a game into multiple smaller games, the canonical form of each of these pieces can be found, and the simplified pieces can be added together.
Being able to do this can dramatically reduce the amount of time required to compute who wins and loses in the larger game.

There has been research in developing the combinatorial game theory of games with rulesets other than normal play, such as mis\`ere games, where the last player to make a move \emph{loses}, and scoring games, where points are assigned to both players and the player with the most points at the end of the game wins (see \cite{universes} for more information on mis\`ere and scoring games).
Recently, Peter Selinger has developed the combinatorial game theory of $\Hex$ \cite{hex_canon}.

First, we give some background on concepts we will use in this paper, including posets and conventions of CGT.
Afterwards we review some properties of $\Rex$ that are highlighted in this thesis, as well as some generalizations of $\Rex$. 
In \cref{Games_over_posets} we discuss how to interpret the outcome of one region of a Rex board using posets.
Then in \cref{Evaluating_games} we define what it means for one game to be better than another game.
We will define this in two ways, called the contextual order and the intrinsic order.
Next, in \cref{Properties_of_Rex_revisited} we discuss special properties of $\Rex$, such as premotivity and $*$-antimonotonicity.
Afterwards in \cref{Intrinsic_extrinsic_united} we prove that the contextual order is equivalent to the intrinsic order when restricted to a special class of games.
Finally, in \cref{Canonical_forms} we prove that every finite game within this special class of games has a unique canonical form, i.e., a unique smallest game that is equivalent to it.

\section{How to play Hex and Rex}
\subsection{Hex}\label{how_to_play_Hex}
\begin{figure}[ht]
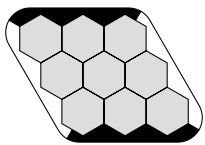

  \centering
  \begin{hexboard}[scale=0.8]
    \rotation{-30}
    \board(3,3)
  \end{hexboard}
  \caption{An empty Hex (or $\Rex$) board.}
  \label{fig:null_board}
\end{figure}
Hex is a board game played on a parallelogram like the one in \cref{fig:null_board}.
The board consists of hexagonal \emph{cells} and four colored terminals, two black and two white.
We will call the players Left and Right.
Left must play one black stone in an unoccupied cell on her turn, and Right must place one white stone on his turn.
The game ends when no moves remain.
If the two black terminals (the top and bottom of the board in \cref{fig:null_board}) are connected by black stones, then Left wins.
If the two white terminals are connected by white stones, then Right wins.

\begin{figure}[ht]
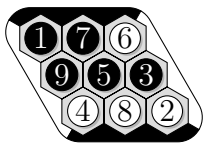

  \centering  
  \begin{hexboard}[scale=0.8]
    \rotation{-30}
    \board(3,3)
    
    \black(1,1)\label{1}
    \white(3,3)\label{2}
    \black(3,2)\label{3}
    \white(1,3)\label{4}
    \black(2,2)\label{5}
    \white(3,1)\label{6}
    \black(2,1)\label{7}
    \white(2,3)\label{8}
    \black(1,2)\label{9}
  \end{hexboard}
  \caption{An example game of Hex (or $\Rex$).}
  \label{fig:ex_game}
\end{figure}

\cref{fig:ex_game} shows an example game of Hex.
The stones are labeled in the order in which they
were played (Left started by playing in the upper-left corner, then Right played in the lower-right corner, etc).
With the cells labeled $4$,$8$, and $2$, we see that Right has connected
the two white terminals with white stones and has won.

\subsection{Rex}\label{how_to_play_Rex}
Reverse Hex, often shortened to $\Rex$, is played in the same way as Hex, with Left and Right taking turns placing black and white stones, respectively, until the board is full.
However, at the end of a game of $\Rex$, Left \emph{loses} if her two black terminals are connected by black stones, and Right loses if his two white terminals are connected by white stones.
For an example game of $\Rex$, we can look at \cref{fig:ex_game} again.
Here, we see that Left has won this game because Right's two terminals are connected and Left's two terminals are not.

Note that we have defined the rules of $\Rex$ (and Hex) so that the game only ends when both boards are completely filled. This convention is convenient for the theory we develop in this thesis. In practice, the players will stop playing as soon as one of them has made a connection, because the winner can no longer change after this point.

\chapter{Background}\label{Background}

\section{Posets}\label{Background:Posets}

In this thesis we will use posets as a basis for our game forms. In light of this,  in this section we provide some of the definitions and conventions for posets that we will use.

\begin{definition}\label{def:posets}
  A \emph{partially ordered set}, or \emph{poset}, is a set $A$ with a relation ${\leq_A}\subseteq A\times A$ such that the following hold:
  \begin{enumerate}
  \item (\emph{Reflexivity}) For all $a\in A$, it must be that $a\leq_A a$.
  \item (\emph{Antisymmetry}) Given any two elements $a,b\in A$, if $a\leq_Ab$ and $b\leq_Aa$ then $a=b$.
  \item (\emph{Transitivity}) Given any three elements $a,b,c\in A$, if $a\leq_Ab$ and $b\leq_Ac$ then $a\leq_Ac$.
  \end{enumerate}
\end{definition}
When it is clear what poset we are dealing with, we will omit the $_A$ and write $a\leq b$.

Let $A$ be a poset, and let $A'$ be a subset of $A$.
Then $A'$ forms a poset with the induced order, i.e., ${\leq_{A'}}={\leq_A}\cap {(A'\times A')}$.

Let $A$ and $B$ be two posets.
The \emph{Cartesian product} $A\times B$ forms a poset with the componentwise order: namely,
\[(a_1,b_1)\leq_{A\times B}(a_2,b_2)\iff a_1\leq_A a_2\text{ and }b_1\leq_B b_2.\]
In a poset, we write $x<y$ when $x\leq y$ and $x\neq y$.

A \emph{preordered set} $P$ is defined similarly to how a poset is defined, but the relation on a preordered set does not need to be antisymmetric.
This means that there can be two elements $x,y\in P$ such that $x\leq y$ and $y\leq x$ and $x\neq y$.
Whenever $P$ is a preordered set, we can define an equivalence relation on $P$.
We say $x\inteq y$ whenever $x\leq y$ and $y\leq x$.
Then the quotient set $P/{\simeq}$ is partially ordered.

\subsection{Monotone functions}
\begin{definition}\label{def:monofun}
  Let $A$ and $B$ be posets and let $f$ be a function from $A$ to $B$.
  We say $f$ is \emph{monotone} if it preserves the order, i.e.,
  \[\forall x,y \in A, \,\,\,\, x\leq_A y\Rightarrow f(x)\leq_Bf(y).\] \end{definition}
Two posets $A$ and $B$ are \emph{isomorphic} to each other if there exist monotone functions $f\colon A\rightarrow B$ and $g\colon B\rightarrow A$ such that $f$ and $g$ are mutual inverses. We denote this as $A\cateq B$.

\subsection{Some common posets}\label{subsec:common_posets}
Given any set $X$, we use $2^X$ to denote the set of \emph{subsets of $X$}.
This forms a poset under the subset relation $\subseteq$.

The boolean poset $\mathbb{B}$ has two elements, $\bot$ and $\top$, with $\bot\leq\top$.
We sometimes think of $\bot$ and $\top$ as the boolean values \emph{false} and \emph{true}, respectively.

The poset $\deadset$ has only one element, and we refer to its only element as $0$.
When dealing with posets, $\deadset$ is the multiplicative identity with respect to the Cartesian product, i.e.,  $A\times\deadset$ is isomorphic to $A$ for all posets $A$.
Because of this, we often identify $A\times\deadset$ with $A$.

For two posets $A$ and $B$, let $\hom(A,B)$ be the set of monotone functions from $A$ to $B$.
We can turn $\hom(A,B)$ into a poset: we say $f\leq_{\hom(A,B)}g$ if $f(x)\leq_B g(x)$ for all $x$ in $A$.

Later on we will frequently refer to the poset $\hom(A,\bool)$.
We will refer to it as $\fun{A}$.

Every poset $A$ mentioned throughout this thesis will also have a \emph{greatest} and \emph{least} element.
We denote the greatest element, called \emph{top}, by $\top$, and we denote the least element, called \emph{bottom}, by $\bot$.
For any $x$ in $A$, we have $\bot\leq x \leq \top$.

Given a poset $A$, we define the poset $\op{A}$ to be the set $A$ with the ordering $b\leq_{\op{A}} a$ whenever $a\leq_A b$.
When convenient, we may write $\op{a}$ and $\op{b}$ to represent elements of $\op{A}$, so that $\op{b}\leq \op{a}$ if and only if $a\leq b$.
Note that $\op{(A\times B)}=\op{A}\times\op{B}$.

If $f:A\rightarrow B$ is a monotone function, then we define $\op{f}\colon\op{A}\to\op{B}$ to be the function such that $\op{f}(\op{x})=\op{f(x)}$ for all $x$ in $A$.
\subsection{Some important monotone functions}\label{Background:mono_fun}

We would like to define a function $c\colon A\times A\rightarrow\bool$ that compares two elements of a poset, i.e., $c(x,y)=\top$ if and only if $x\leq y$.
Unfortunately, $c$ is not a monotone map.
For example, let $y$ and $z$ be elements of $A$ such that $y<z$.
Then $(y,y)\leq (z,y)$, but $c(y,y)=\top$ and $c(z,y)=\bot$.

To fix this, we need our map to go from $\op{A}\times A$ to $\bool$.
We define $\lambda$ as follows: $\lambda(\op{x},y)=\top$ if and only if $x\leq_A y$.
It can be shown that $\lambda$ is a monotone function. 

Similarly, we define $\rho\colon A\times \op{A}\to\bool$ by $\rho(x,\op{y})=\top$ if and only if $x\nleq y$.
Note that, because $\bool$ is isomorphic to $\op{\bool}$, the function $\rho$ is essentially the same as $\op{\lambda}$.

We define $\emph{and}$ and $\emph{or}$ to be the logical operations as maps from $\bool\times\bool$ to $\bool$.
These are both monotone functions.

\section{Combinatorial Game Theory}\label{sec:CGT}
Combinatorial game theory (CGT) is the study of perfect strategy games: two player games without hidden information or randomness.
At its core, CGT is about developing ways to break a game into small parts, ``simplify'' them, and add the results together.
The definition we give excludes hidden information games like Battleship and rock-paper-scissors, as well as games based on chance like Poker, Snakes and Ladders, and Ludo.
Another requirement that our games must satisfy is the infinite descending game condition.
This states that all games must end after a finite number of moves.
CGT has been used to study popular games such as \textsc{Chess}, \textsc{Checkers (Draughts)}, \textsc{Mancala}, \textsc{Go} and  \textsc{Tic-Tac-Toe}.
Some combinatorial games, such as variations of \textsc{NIM}, have been studied since the early 1900's \cite{Bouton}.
The combinatorial game theory developed herein is derived from Berlekamp, Conway, and Guy's research on \emph{normal play} games in the 1970's \cite{ww2}, as well as the work of Selinger in 2022 \cite{hex_canon}.

In CGT, games are often classified by how a winner is decided.
These include normal play games (in which a player loses if they cannot move on their turn),
mis\`ere play (in which a player \emph{wins} if they cannot make a move on their turn),
and scoring games (in which the winner is whoever gets the highest score).
Normal play games, also known as Conway games, have been extensively studied because of their nice structure.
In the following paragraphs we will use normal play games to showcase important concepts in CGT, and we outline how this may be adapted to non-normal play games like $\Rex$. The main references we use for this section are \cite{ww2}, \cite{ONAG}, and \cite{Intro_to_CGT}.

In CGT, the word \emph{ruleset} refers to the rules and structures of a game, while a \emph{position} refers to a state that is achievable under the given ruleset.
The term \emph{game form}, often just called a \emph{game}, is a recursive way of describing a position.
As mentioned before, all combinatorial games are played by two players. There are many conventions for what these two players are called: Alice and Bob, Louise and Richard, Blue and Red, etc. In this thesis we name these players \emph{Left} and \emph{Right}.

In every position, there is a set of moves that Left can make and a set of moves that Right can make. The \emph{game form} $G$ of a position $P$ is a pair of sets $\gc{L}{R}$, where every $X\in L$ is the game form of a position Left could reach if they made a move in the position $P$.
The set $R$ is defined similarly for Right.
We call every element of $L$ a \emph{left option of $G$}, and the elements of $R$ are the \emph{right options of $G$}.
We say that $G'$ is a \emph{follower} of $G$ if $G'$ is $G$ itself, any option of $G$, or (recursively) if $G'$ is a follower of some option of $G$. Informally, a follower is any potential state of the game reachable from $G$.

In this thesis we will use a few notational conventions that are common in CGT.
For instance, we denote a game $\gc{\{W,X\}}{\{Y,Z\}}$ by $\gc{W,X}{Y,Z}$, separating the individual elements with commas.
When convenient, we write $G=\gc{G^L}{G^R}$, where $G^L$ denotes a typical left option of $G$ and $G^R$ denotes a typical right option of $G$. We use this notation even when $G$ has more than one left or right option.

The simplest Conway game is the game with no left options and no right options.
This game, called $0$, can be represented as $0=\gc{}{}$.

There are many ways of adding two games in CGT.
In this thesis we will only be using the disjunctive sum of games.
When we take the disjunctive sum of two normal play games $G$ and $H$, we define this as \[G+H\identical\gc{G^L+H,G+H^L}{G^R+H,G+H^R}.\]
Informally, we think of this as playing both games at the same time. On Left's turn they chose to make a move in $G$ or make a move in $H$, and similarly for Right.

The game $0$ acts as an identity for the disjunctive sum: For any normal play game $G$, the game $G+0$ is the same as the game $G$.

\cref{fig:big_rex} depicts a $\Rex$ position.
Say we wanted to determine if Left has a first player winning strategy from this point in the game.
There are $10$ empty cells left on the board.
Left starts with $10$ options of where to move, then Right has $9$ options remaining, then Left has $8$, and so on, meaning the number of ways the game could play out is $10!$, or $3,628,800$.
We will adapt ideas from combinatorial game theory to make analyzing games like this simpler.

First, notice that playing the game in \cref{fig:big_rex} is equivalent to playing the disjunctive sum of the three games with gluing instructions (see \cref{fig:split_big_rex}).
The games $G_1$, $G_2$, and $G_3$ have $3$, $3$, and $4$ empty cells, respectively.
It is clear that $4!+3! +3!$, or $36$, is much smaller than $10!$.
By analyzing $G_1,G_2,$ and $G_3$ separately and adding/gluing the results together, we can reduce computation time drastically. 
\begin{figure}[ht]
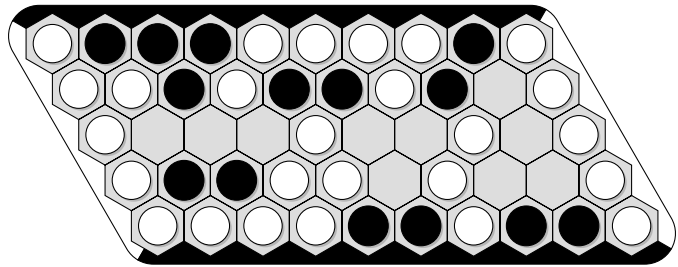

  \centering
  \begin{hexboard}
    \rotation{-30}
    \board(10,5)
    \white(1,1)\black(2,1)\black(3,1)\black(4,1)\white(5,1)\white(6,1)\white(7,1)\white(8,1)\black(9,1)\white(10,1)
    \white(1,2)\white(2,2)\black(3,2)\white(4,2)\black(5,2)\black(6,2)\white(7,2)\black(8,2)\white(10,2)
    \white(1,3)\white(5,3)\white(8,3)\white(10,3)
    \white(1,4)\black(2,4)\black(3,4)\white(4,4)\white(5,4)\white(7,4)\white(10,4)
    \white(1,5)\white(2,5)\white(3,5)\white(4,5)\black(5,5)\black(6,5)\white(7,5)\black(8,5)\black(9,5)\white(10,5)
  \end{hexboard}
  
  \caption[A large game of $\Rex$.]{This game of $\Rex$ is already too big to analyze by hand...}
  \label{fig:big_rex}
\end{figure}

\begin{figure}[ht]
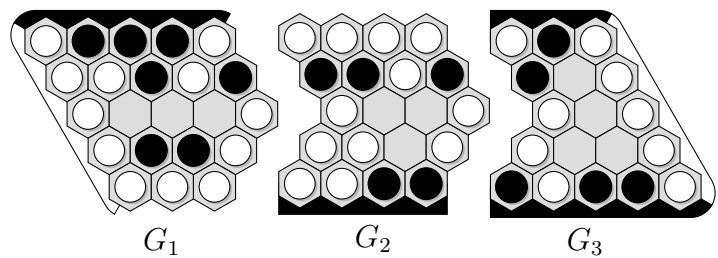

  \centering
  \[
    \begin{hexboard}[scale=0.80, baseline=(current bounding box.center)]
      \rotation{-30}
      \foreach \i in {1,...,5} { \hex(\i,1) }
      \white(1,1)\black(2,1)\black(3,1)\black(4,1)\white(5,1)
      \foreach \i in {1,...,5} { \hex(\i,2) }
      \white(1,2)\white(2,2)\black(3,2)\white(4,2)\black(5,2)
      \foreach \i in {1,...,5} { \hex(\i,3) }
      \white(1,3)\white(5,3)
      \foreach \i in {1,...,4} { \hex(\i,4) }
      \white(1,4)\black(2,4)\black(3,4)\white(4,4)
      \foreach \i in {1,...,3} { \hex(\i,5) }
      \white(1,5)\white(2,5)\white(3,5)
      \edge[\nw\noobtusecorner] (1,1)(5,1)
      \edge[\sw] (1,1)(1,5)
      \node at \coord(1,6.5) {$G_1$};
    \end{hexboard} 
    \begin{hexboard}[scale=0.80, baseline=(current bounding box.center)]
      \rotation{-30}
      \foreach \i in {5,...,8}{ \hex(\i,1) }
      \white(5,1)\white(6,1)\white(7,1)\white(8,1)
      \foreach \i in {5,...,8}{ \hex(\i,2) }
      \black(5,2)\black(6,2)\white(7,2)\black(8,2)
      \foreach \i in {5,...,8}{ \hex(\i,3) }
      \white(5,3)\white(8,3)
      \foreach \i in {4,...,7}{ \hex(\i,4) }
      \white(4,4)\white(5,4)\white(7,4)
      \foreach \i in {3,...,6}{ \hex(\i,5) }
      \white(3,5)\white(4,5)\black(5,5)\black(6,5)
      \edge[\se\noobtusecorner\noacutecorner](3,5)(6,5)
      \node at \coord(4,6.5) {$G_2$};
    \end{hexboard} 
    \begin{hexboard}[scale=0.80, baseline=(current bounding box.center)]
      \rotation{-30}
      \foreach \i in {8,...,10}{ \hex(\i,1) }
      \white(8,1)\black(9,1)\white(10,1)
      \foreach \i in {8,...,10}{ \hex(\i,2) }
      \black(8,2)\white(10,2)
      \foreach \i in {8,...,10}{ \hex(\i,3) }
      \white(8,3)\white(10,3)
      \foreach \i in {7,...,10}{ \hex(\i,4) }
      \white(7,4)\white(10,4)
      \foreach \i in {6,...,10}{ \hex(\i,5) }
      \black(6,5)\white(7,5)\black(8,5)\black(9,5)\white(10,5)
      \edge[\nw\noacutecorner](8,1)(10,1)
      \edge[\ne](10,1)(10,5)
      \edge[\se\noobtusecorner](6,5)(10,5)
      \node at \coord(7,6.5) {$G_3$};
    \end{hexboard}
  \]
  \caption[Large game of $\Rex$, piecemeal.]{... but breaking the game into smaller pieces aids computation.}
  \label{fig:split_big_rex}
\end{figure}

Often in this paper, we will use a proof technique known as Conway induction.
Below we will state the Conway induction theorem for normal play games:
\begin{theorem}[Conway induction]
  Let $P$ be a property that a game can have.
  Suppose that, given a game $G$, if $P(G')$ is true for all left options and all right options $G'$ of $G$, then $P(G)$ is true.
  Then $P(G)$ is true for all games $G$.
\end{theorem}
\begin{proof}
  Let $P$ be as above, and suppose there is a game $G$ such that $P(G)$ is false.
  Then there must be some left option or right option $G^{(1)}$ of $G$ such that $P(G^{(1)})$ is false.
  Then there must be some left or right option $G^{(2)}$ of $G^{(1)}$ such that $P(G^{(2)})$ is false.
  Continuing like this we find that for any $n$ there exists a descendant $G^{(n+1)}$ of $G^{(n)}$, which violates the infinite descending game condition that we placed on normal play games.
  This is a contradiction, so $P(G)$ must be true.
\end{proof}

A more general version of this theorem can be found in \cite{Intro_to_CGT}.
In \cref{Games_over_posets} we will prove a version of this theorem which is more appropriate for $\Rex$.

In CGT, we orient everything through the eyes of Left.
Thus, when we say something is better, we mean better for Left.
Sometimes we want to make a claim from Right's perspective.
When this is necessary, we use the the dual of the game, denoted by $\op{G}$, in which Left and Right switch sides.
\begin{definition}[Dual of a game]\label{def:dual}
  Let $G\identical\gc{G^L}{G^R}$ be some game.
  The \emph{dual} of the game $G$ is called $\op{G}$, and is defined as $\op{G}=\gc{\op{(G^R)}}{\op{(G^L)}}$.
\end{definition}

Given any concept, property, or proof in CGT, there is a \emph{dual} concept in which the roles of Left and Right are swapped. For example, take the definition of Left having a first and second player winning strategy:

\begin{definition}[First and second player winning strategy for Left]
  We define a first (second) player winning strategy for Left recursively as follows:
  \begin{itemize}
  \item Left has a first player winning strategy in a game $G$ if Left has an option $G^L$ of $G$ such that Left has a second player winning strategy in $G^L$.
  \item Left has a second player winning strategy in a game $G$ if, for every right option $G^R$ of $G$, Left has a first player winning strategy in $G^R$.
  \end{itemize}
\end{definition}

This definition begets a dual definition in which Left and Right swap places:
\begin{itemize}
\item \emph{Right} has a first player winning strategy in a game $G$ if \emph{Right} has an option $G^R$ of $G$ such that \emph{Right} has a second player winning strategy in $G^R$, and
\item \emph{Right} has a second player winning strategy in a game $G$ if, for every \emph{Left} option $G^L$ of $G$, \emph{Right} has a first player winning strategy in $G^L$.
\end{itemize}
In this thesis we will not explicitly state the dual of a concept/property/proof when its dual follows naturally.

Note that Left has a first player winning strategy if and only if Right does not have a second player winning strategy, which can be proved using Conway induction.

In CGT we care about the \emph{outcome class} of a game, i.e., who wins under optimal play when going first and who wins when going second. We break this up into left and right parts:
\begin{definition}[Left and right outcome classes]\label{def:LR_outcome_class}
  The left outcome class of $G$, denoted by $o_L(G)$, is defined as follows:
  \begin{itemize}
  \item $o_L(G)=\top$ when Left has a winning strategy with Left going first, and
  \item $o_L(G)=\bot$ when Right has a winning strategy with Left going first.
  \end{itemize}
  The right outcome class of $G$, denoted by $o_R(G)$, is defined as follows:
  \begin{itemize}
  \item $o_R(G)=\top$ when Left has a winning strategy with Right going first, and
  \item $o_R(G)=\bot$ when Right has a winning strategy with Right going first.
  \end{itemize}
\end{definition}

To know who wins going first and who wins going second, it is sufficient to know $o_L(G)$ and $o_R(G)$. We define outcome classes as follows:

\begin{definition}[Outcome class]\label{def:outcome_class}
  The \emph{outcome class} of a game $G$, denoted by $o(G)$, is
  \begin{itemize}
  \item $o(G)=\gl$ if $(o_L(G),o_R(G))=(\top,\top)$, i.e., Left has a first and second player winning strategy in $G$.
  \item $o(G)=\gn$ if $(o_L(G),o_R(G))=(\top,\bot)$, i.e., Left and Right both have a first player winning strategy in $G$.
  \item $o(G)=\gp$ if $(o_L(G),o_R(G))=(\bot,\top)$, i.e., Left and Right both have a second player winning strategy in $G$.
  \item $o(G)=\gr$ if $(o_L(G),o_R(G))=(\bot,\bot)$, i.e., Right has a first and second player winning strategy in $G$.
  \end{itemize}
\end{definition}

We can think of $o(G)$ as a function which sends a game to $\bool\times\bool$. Then we can turn the outcome classes into a poset, as in \cref{fig:poset_outcomes}.
This gives us the following lemma:
\begin{lemma}
  Let $G$ and $H$ be games.
  Then $o(G)\leq o(H)$ if and only if both of the following hold:
  \begin{itemize}
  \item if Left has a first player winning strategy in $G$ then Left has a first player winning strategy in $H$, and
  \item if Left has a second player winning strategy in $G$ then Left has a second player winning strategy  in $H$.
  \end{itemize}
  \label{strategy_outcome}
\end{lemma}

\begin{figure}[h]
  \centering
  $\begin{tikzpicture}[baseline=(current bounding box.center)]
    \node (L) at (0,1) {$\gl$};
    \node (N) at (-1,0) {$\gn$};
    \node (P) at (1,0) {$\gp$};
    \node (R) at (0,-1) {$\gr$};
    \path (L) edge [-] (N);
    \path (L) edge [-] (P);
    \path (R) edge [-] (N);
    \path (R) edge [-] (P);
  \end{tikzpicture}$
  \caption{The poset of outcomes $o(G)$.}
  \label{fig:poset_outcomes}
\end{figure}
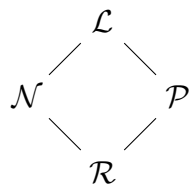

If $G$ and $H$ are normal play games, we say $G\leq H$ if, for all normal play games $X$, \[o(G+X)\leq o(H+X).\]

It turns out that $G\leq H$ if and only if Left has a second player winning strategy in the game $\op{G}+H$. We write this recursively as follows:

\begin{definition}[Order on normal play games]\label{normal_intrinsic}
  We define the two relations, $\leq$ and $\tri$, on normal play games by mutual induction.
  
  We say $G\leq H$ when Left has a second player winning strategy in $\op{G}+H$, i.e., when all of the following hold:
  \begin{enumerate}
  \item $G^L\tri H$ for all $G^L$, and
  \item $G\tri H^R$ for all $H^R$.
  \end{enumerate}

  We say $G\tri H$ when Left has a \emph{first} player winning strategy in $\op{G}+H$, i.e., when at least one of the following hold:
  \begin{enumerate}
  \item there exists some $G^R$ such that $G^R\leq H$, or
  \item there exists some $H^L$ such that $G\leq H^L$.
  \end{enumerate}  
\end{definition}

We say that $G$ and $H$ are \emph{equivalent} if $G\leq H$ and $H\leq G$.
This is not the same as saying two games are identical, by which we mean they are literally the same.

Note that in CGT, it is common to say games are ``equal'' when they are merely equivalent. We will avoid this convention for the sake of clarity.

\chapter{Some properties of \texorpdfstring{$\Rex$}{Rex} and related games}\label{Rex_concrete}

The results about $\Rex$ that this thesis provides will also apply to more general games.
Here we will discuss some special properties of $\Rex$ and we will describe some generalizations of the game that also have these properties.
The rules of $\Rex$ were given in \cref{how_to_play_Rex}. 
\section{Games like \texorpdfstring{$\Rex$}{Rex}}
One generalization of $\Rex$ called the \emph{reverse Shannon game} is played on a graph.
The rules are as follows:

\begin{definition}[Reverse Shannon game]
  Let $B$ be a finite graph with two distinguished vertices $t_N$ and $t_S$.
  The vertices $t_N$ and $t_S$ are the called the \emph{north} and \emph{south} terminals, respectively.
  The other vertices are called \emph{cells}.
  On a player's turn they must choose one uncolored cell and color it.
  Left will color cells black, and Right will color cells white.
  At the end of the game, Left wins if and only if there is no black path between the north and south terminals, i.e., no path from $t_N$ to $t_S$ that only touches black cells.
\end{definition}

Here is an example of a reverse Shannon game's graph:

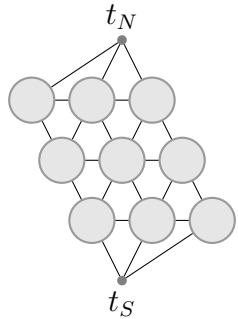
\[
  \begin{tikzpicture}
    [scale=0.8,
    cell/.style={circle,draw=black!40,fill=black!10,thick,
      inner sep=0pt,minimum size=6mm},
    terminal/.style={circle,draw=black!50,fill=black!50,thick,
      inner sep=0pt,minimum size=1mm}]
    
    \node (N) at (0,4) [terminal] {};
    \node [above] at (0,4) {$t_N$};
    
    \node (a) at (-1.5,3) [cell] {};
    \node (b) at (-0.5,3) [cell] {};
    \node (c) at ( 0.5,3) [cell] {};
        
    \node (d) at (-1,2) [cell] {};
    \node (e) at (0,2) [cell] {};
    \node (f) at (1,2) [cell] {};

    \node (g) at (-0.5,1) [cell] {};
    \node (h) at (0.5,1) [cell] {};
    \node (i) at (1.5,1) [cell] {};
    
    \node (S) at (0,0) [terminal] {};
    \node [below] at (0,0) {$t_S$};

    \path (N) edge (a)
    edge (b)
    edge (c);

    \path (a) edge (b)
    edge (d);
    \path (b) edge (c)
    edge (d)
    edge (e);
    \path (c) edge (e)
    edge (f);
    
    \path (d) edge (e)
    edge (g);
    \path (e) edge (f)
    edge (g)
    edge (h);
    \path (f) edge (h)
    edge (i);

    \path (g) edge (h);
    \path (h) edge (i);
    
    \path (S) edge (g)
    edge (h)
    edge (i);
  \end{tikzpicture}
\]

$\Rex$ is a special case of the reverse Shannon game.
In fact, the above example is analogous to $3\times 3$ $\Rex$.

There is a further generalization of $\Rex$ and of the reverse Shannon game, which is the \emph{antimonotone set coloring game}.
Recall that $\bool$ is the set of booleans, and that  $2^X$ is the poset formed by the subsets of some set $X$.

\begin{definition}[Antimonotone set coloring game]\label{setColoringGame}
  A \emph{set coloring game} is given by a set $X$ --- whose elements we call \emph{cells} --- and a function $f\colon 2^X\to\bool$, which is called a \emph{payoff function}.
  On a player's turn they must choose one uncolored cell.
  Left will color cells black, and Right will color cells white.
  Let $L$ be the set of black cells at the end of the game.
  Left wins if and only if $f(L)=\top$.

  A set coloring game is \emph{antimonotone} if $f$ is an antimonotone function; in other words, for all subsets $L$ and $L'$ of $X$, if $L\subseteq L'$ then $f(L')\leq f(L)$. 
\end{definition}

Every reverse Shannon game (and therefore every $\Rex$ game) is an antimonotone set coloring game: for all subsets $L$ of the cells $X$, $f(L)=\top$ if and only if there is no path from $t_N$ to $t_S$ passing only through cells in $L$.
Every reverse Shannon game is antimonotone as well.
However, not every antimonotone set coloring game can be expressed as a reverse Shannon game.
For example, the game in which $X$ is a set with 3 elements and $f(L)=\top$ if and only if ${|L|} < 2$ cannot be expressed as a reverse Shannon game: for every pair of cells, one must be connected to $t_N$ and the other must be connected to $t_S$. This cannot be accomplished without one of the three cells being connected to both $t_N$ and $t_S$, contradicting the assumption that Left wins if and only if ${|L|}<2$.

The combinatorial game theory developed in this paper applies to the class of antimonotone set coloring games.
Most of the examples will be given in terms of $\Rex$, but all of the statements remain true when considering this more general class of games.

\section{Properties of \texorpdfstring{$\Rex$}{Rex}}
We now introduce, by examples, three properties of $\Rex$ (and more generally, of antimonotone set coloring games): parity, $*$-antimonotonicity, and premotivity.
These properties will play an important role in the combinatorial game theory we develop in later chapters.

\subsection{Parity}

We say that a $\Rex$ position is \emph{even} if there is an even number of empty cells, and \emph{odd} if there is an odd number of empty cells.
The same definition applies to any set coloring game.
Notice that if a player makes a move in an even position, the resulting position is odd, and vice versa.
\subsection{\texorpdfstring{$*$}{*}-Antimonotonicity}\label{concrete_*_antimonotonicity}

If $G$ and $H$ are positions in a set coloring game, we say that $G\leq_o H$ if $o(G)\leq o(H)$,
i.e., if whenever Left has a first (respectively, second) player winning strategy in $G$,
then Left also has a first (respectively, second) player winning strategy in $H$.
We say that $G\inteq_o H$ if $G\leq_o H$ and $H\leq_o G$.

{
  
  In $\Hex$, it is well known that Left always prefers a black stone to a white stone. It is convenient for us to represent this pictographically like so:
  \[\whitestone*{} \leq_o \blackstone*{}.\]
  We call this property \emph{monotonicity}. $\Hex$ also satisfies a stronger property, which we call \emph{strong monotonicity}:
  In $\Hex$, Left prefers a black stone in a cell over the cell being empty, and Left prefers a cell to be empty rather than the cell to have a white stone (see \cite{hex_canon} for more information):
  \[\whitestone*{} \leq_o \inlinehex*{} \leq_o \blackstone*{}.\]
}

Naturally $\Rex$ is \emph{antimonotone}, i.e., in $\Rex$, Left always prefers a white stone to a black stone:
  \[\blackstone*{} \leq_o \whitestone*{}.\]

  Because $\Rex$ is the reverse of $\Hex$, it is tempting to think that $\Rex$ is also strongly \emph{anti}-monotone, i.e., 
\[\blackstone*{} \leq_o \inlinehex*{} \leq_o \whitestone*{}.\]

Perhaps surprisingly, this turns out to be false. For example, take the following game of $\Rex$:

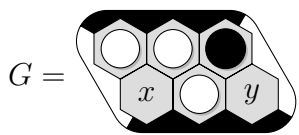
\[G=
  \begin{hexboard}[baseline={($(current bounding box.center)-(0,1ex)$)}]
    \rotation{-30}
    \board(3,2)
    \white(1,1)\white(2,1)\black(3,1)
    \cell(1,2)\label{$x$}\white(2,2)\cell(3,2)\label{$y$}
  \end{hexboard}\]

The position $G$ has outcome class $\gn$: the player who goes first plays in cell $x$, and the other player must then play in cell $y$ and lose.
On the other hand, let $G'$ be the same as position $G$, but where $x$ is already occupied by a black or white stone, so that $y$ is the only empty cell in $G'$.
Then $G'$ has outcome class $\gp$: whoever goes first loses.
This shows that
\[\inlinehex*{$x$}\inlinehex*{$y$}\nleq_o \whitestone*{$x$}\inlinehex*{$y$}.\]
In particular, it shows that $\Rex$ is not strongly antimonotone.

While we do not have strong antimonotonicity, we \emph{do} have that \emph{two} black stones are worse for Left than \emph{two} empty cells, and two empty cells are worse for Left than two white stones, i.e.,
\begin{equation}
  \label{eq:SAM}
  \blackstone*{}\blackstone*{} \leq_o \inlinehex*{}\inlinehex*{} \leq_o \whitestone*{}\whitestone*{}.
\end{equation}

The proof of this relies on the parity of the position not changing, along with the fact that $\Rex$ is antimonotone.
We will prove that $\blackstone*{}\blackstone*{} \leq_o \inlinehex*{}\inlinehex*{}$, as the proof of $\inlinehex*{}\inlinehex*{} \leq_o \whitestone*{}\whitestone*{}$ is dual.
Assume that Left has a (first or second player) winning strategy in a $\Rex$ position with black stones in two cells, labeled $\blackstone*{$a$}$ and $\blackstone*{$b$}$.
We want to show that Left has a (first/second player respectively) winning strategy in the position including $\inlinehex*{$a$}\inlinehex*{$b$}$.
In this position, Left simply pretends both cells have black stones and follows their original strategy.
If Right ever plays in one of these two cells, Left plays in the other cell, and their strategy continues as if both cells had black stones.
If, on Left's turn, $\inlinehex*{$a$}$ and $\inlinehex*{$b$}$ are the only empty cells remaining, then Left plays in one and Right must play in the other.
At the end of the game the board will be exactly the same as if Left had started with $\blackstone*{$a$}$ and $\blackstone*{$b$}$, except one of these stones will be white instead.
Because $\Rex$ is antimonotone, Left wins anyway.


This fact implies a property of antimonotone set coloring games that is reminiscent of the strong monotonicity we see in $\Hex$.
First, in a set coloring game, we say that a cell is \emph{dead} if, no matter how you fill the board, changing the color of that cell does not affect the winner.

Dead cells, denoted by $*$, are useful in antimonotone games: playing in a dead cell is like passing your turn, forcing your opponent to move instead.
If there is an even number of dead cells in a position, they can be ignored.
To see why, take a $\Rex$ position with a cell that is not dead, which we will call \inlinehex*{$x$}, along with two dead cells, \inlinehex*{$*$} and \inlinehex*{$*$}.

By \cref{eq:SAM} we have that
\[\inlinehex*{$x$}\blackstone*{$*$}\blackstone*{$*$} \leq_o
  \inlinehex*{$x$}\inlinehex*{$*$}\inlinehex*{$*$} \leq_o
  \inlinehex*{$x$}\whitestone*{$*$}\whitestone*{$*$}.\]

The color of a filled dead cell does not matter, so we can remove the filled stones to obtain:
\[ \inlinehex*{$x$} \leq_o \inlinehex*{$x$}\inlinehex*{$*$}\inlinehex*{$*$} \leq_o \inlinehex*{$x$},\]

or in other words,
\begin{equation}
  \label{eq:2star}
  \inlinehex*{$x$} \inteq_o \inlinehex*{$x$}\inlinehex*{$*$}\inlinehex*{$*$}.
\end{equation}
While we do not have that $\blackstone*{}\leq_o \inlinehex*{} \leq_o \whitestone*{}$ in $\Rex$, we do have the following property, which we call \emph{$*$-antimonotonicity}:
\[ \blackstone*{$x$} \inlinehex*{$*$} \leq_o\inlinehex*{$x$} \leq_o\whitestone*{$x$} \inlinehex*{$*$}. \]

Informally, we are saying that the only time when Left would prefer an empty cell to a white stone is when Left wants to force Right to play somewhere else by limiting their options.


Proof of $*$-antimonotonicity: take some empty cell in a position, denoted by \inlinehex*{$x$}.
Then we have the following:
\[\blackstone*{$x$}\inlinehex*{$*$}\inteq_o
  \blackstone*{$x$}\inlinehex*{$*$}\blackstone*{$*$} \leq_o
  \inlinehex*{$x$}\inlinehex*{$*$}\inlinehex*{$*$}\inteq_o
  \inlinehex*{$x$}.\]

The first equivalence holds because filled dead cells are inconsequential.
The middle inequality holds by \cref{eq:SAM}, and the last equivalence holds by \cref{eq:2star}.
The proof of $\inlinehex*{$x$} \leq_o\whitestone*{$x$} \inlinehex*{$*$}$ is dual.


The game of $\Rex$ that we saw earlier has a dead cell as well:

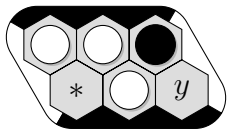
\[\begin{hexboard}
    \rotation{-30}
    \board(3,2)
    \white(1,1)\white(2,1)\black(3,1)
    \cell(1,2)\label{$*$}\white(2,2)\cell(3,2)\label{$y$}
  \end{hexboard}\label{dead_cell}\]
The empty cell labeled $*$ has no way of reaching the northern black terminal, and all of the white stones around the cell are connected to each other already.
Because of this the game's outcome is unaffected by what color stone is in this cell at the end of the game, and so it is dead.

\subsection{Premotivity}\label{rex_premotivity}

The discussion in this subsection applies to all set coloring games, whether they are antimonotone or not.
\begin{lemma}[Move pre-selection lemma]\label{cornering}
  In an even position with at least two empty cells, if the first player refuses to play in some cell $\inlinehex*{$x$}$ then their opponent will be forced to play in $\inlinehex*{$x$}$ by the end of the game.
\end{lemma}
\begin{proof}
  Without loss of generality, we consider the case where Left goes first.
  
  Suppose Left refuses to play in an arbitrary cell $\inlinehex*{$x$}$.
  Since Left went first in an even position, Right will make the last move of the game, and so Right will be forced to play in $\inlinehex*{$x$}$ at the end if they have not already.
\end{proof}

This lemma tells us a lot about the power of parity in a game of $\Rex$.
This lemma leads us to another property that set coloring games have:

\begin{proposition}[Lookahead property]\label{prop_x}
  If Left has a second player winning strategy in an even position, then Left also has a first player winning strategy.
\end{proposition}
\begin{proof}
  We note that, when analyzing a position with no empty cells, Left has a second player winning strategy if and only if Left has a first player winning strategy.

  Suppose that Left has a second player winning strategy in an even game.
  Then Left can win as first player by doing the following.
  First, Left picks an arbitrary empty cell $\inlinehex*{$x$}$ on the board and pretends that Right has played there.
  Left then plays what their second player winning strategy suggests in response to this imagined move by Right.
  Of course, Left's strategy never requires her to play in $\inlinehex*{$x$}$ because it assumes Right has already played there.
  By \cref{cornering} she forces Right to play to $\whitestone*{$x$}$ at some point in the game.
  
  Once Right has played $\whitestone*{$x$}$, the resulting even position will be as if Left had in fact gone second with Right playing $\whitestone*{$x$}$ first. If the new position has no empty cells, then she has won via their second player winning strategy. If this new position has empty cells remaining, then Left picks another arbitrary empty cell and repeats this process until the game ends.\end{proof}

Many of the results in this thesis will require that our games satisfy \cref{prop_x}.
Recall that, in combinatorial game theory, there is a well known concept called \emph{the copycat strategy}.
Here is a story that illustrates how the copycat strategy works.
Suppose Alice is playing two games of $\Rex$ simultaneously, one against Bob and one against Charlie.
Both games are played starting from identical positions.
In her game against Bob she is playing the black stones, and in her game against Charlie she plays the white stones.
Alice waits for Charlie to make his first move.
After this, because both positions are identical, Alice can copy the move Charlie made onto Bob's board.
This means that the new positions are identical after Alice copies this move.
Whenever Bob makes a move Alice copies it in Charlie's game, and whenever Charlie makes a move Alice copies it in Bob's game.
By following this strategy Alice can force the two boards to be identical at the end of play.
Then Bob wins if and only if Charlie loses, so Alice must have won at least one of the games.

Note that Alice's copycat strategy does not require Bob and Charlie to be uncoordinated.
The copycat strategy works even if Bob and Charlie play as a team --- which we will refer to as \emph{Charboblie} --- to defeat Alice.
Therefore, we can think of Alice as playing a single set coloring game against Charboblie.

More precisely, let us write $G$ to refer to the set coloring game Bob plays against Alice,
and $\op{G}$ to refer to the set coloring game that Charlie plays against Alice
(which is the same as $G$ except the colors have been switched).
Let $\op{G}+G$ be the set coloring game played on the union of the two boards, in which Alice wins if and only if she can win in $G$ or $\op{G}$ by the end of the game.

Alice's second player copycat strategy still allows her to win in the game $\op{G}+G$.
The games $G$ and $\op{G}$ each have $n$ empty cells, so $\op{G}+G$ must have $2n$ empty cells, meaning the sum is an even position.
By using \cref{prop_x}, we see that Alice has a \emph{first} player copycat strategy as well!

First player copycat strategies are very uncommon in combinatorial games. It is remarkable that all set coloring games have this property. We say that a game $G$ is \emph{premotive} if Left has a first player winning strategy in the game $\op{G}+G$. We will give a more formal definition of premotivity in \cref{premotive}.

\chapter{Games over posets}\label{Games_over_posets}
\section{Evaluating completions of a region}
Almost all definitions in combinatorial game theory are recursive.
This means that, to define what game forms for some region of a $\Rex$ board should look like, we have to start with the base cases: all the ways to completely fill said region.

\begin{figure}[ht]
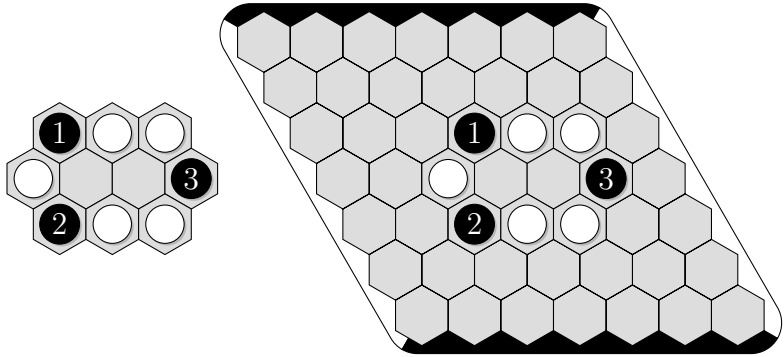

  \centering
  $\begin{hexboard}[baseline=(current bounding box.center)]
    \rotation{-30}
    \hex(4,3)\black(4,3)\label{$1$}\hex(5,3)\white(5,3)\hex(6,3)\white(6,3)
    \hex(3,4)\white(3,4)\hex(4,4)\hex(5,4)\hex(6,4)\black(6,4)\label{$3$}
    \hex(3,5)\black(3,5)\label{$2$}\hex(4,5)\white(4,5)\hex(5,5)\white(5,5)
  \end{hexboard}
  \begin{hexboard}[baseline=(current bounding box.center)]
    \rotation{-30}
    \board(7,7)
    \hex(4,3)\black(4,3)\label{$1$}\hex(5,3)\white(5,3)\hex(6,3)\white(6,3)
    \hex(3,4)\white(3,4)\hex(4,4)\hex(5,4)\hex(6,4)\black(6,4)\label{$3$}
    \hex(3,5)\black(3,5)\label{$2$}\hex(4,5)\white(4,5)\hex(5,5)\white(5,5)    
  \end{hexboard}$
  \caption{A small region embedded in a larger $\Rex$ board.}
  \label{fig:region}
\end{figure}

Take the region of the board shown in \cref{fig:region}.
It has 2 empty cells, meaning there are $2^2=4$ ways of filling the region with black or white stones.
We call these the \emph{completions} of the region.
Looking at \cref{fig:completions}, we see that $\blackstone*{$x$}\blackstone*{$y$}$ is worse for black than $\whitestone*{$x$}\whitestone*{$y$}$: in $\blackstone*{$x$}\blackstone*{$y$}$ the bordering black cells \blackstone*{$1$}, \blackstone*{$2$}, and \blackstone*{$3$} are connected, while in $\whitestone*{$x$}\whitestone*{$y$}$ they are disconnected.
Say we were to embed the region $R$ into any large board $X$ that is already completely filled with black and white stones.
If Left wins in the game $X(R)$ with the completion $\blackstone*{$x$}\blackstone*{$y$}$, then there must not be a path going from the positive terminal to the negative terminal.
When changing both black stones to white stones, this fact remains unchanged.
This is the inspiration for the following order on completions of a region:

\begin{figure}
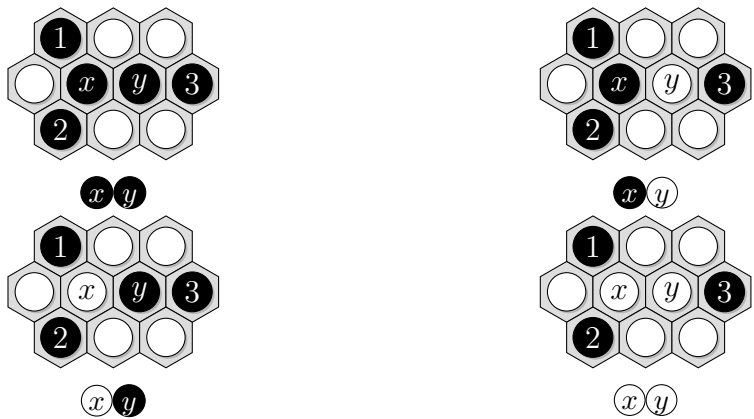

  \begin{subfigure}[t]{0.50\textwidth}
    \centering
    $\begin{hexboard}
      \rotation{-30}
      \hex(4,3)\black(4,3)\label{$1$}\hex(5,3)\white(5,3)\hex(6,3)\white(6,3)
      \hex(3,4)\white(3,4)\hex(4,4)\black(4,4)\label{$x$}\hex(5,4)\black(5,4)\label{$y$}\hex(6,4)\black(6,4)\label{$3$}
      \hex(3,5)\black(3,5)\label{$2$}\hex(4,5)\white(4,5)\hex(5,5)\white(5,5)
    \end{hexboard}$
    \caption*{$\blackstone*{$x$}\blackstone*{$y$}$}
  \end{subfigure}
  \hfill
  \begin{subfigure}[t]{0.50\textwidth}
    \centering
    $\begin{hexboard}
      \rotation{-30}
      \hex(4,3)\black(4,3)\label{$1$}\hex(5,3)\white(5,3)\hex(6,3)\white(6,3)     \hex(3,4)\white(3,4)\hex(4,4)\black(4,4)\label{$x$}\hex(5,4)\white(5,4)\label{$y$}\hex(6,4)\black(6,4)\label{$3$}
      \hex(3,5)\black(3,5)\label{$2$}\hex(4,5)\white(4,5)\hex(5,5)\white(5,5)
    \end{hexboard}$
    \caption*{$\blackstone*{$x$}\whitestone*{$y$}$}
  \end{subfigure}
  \hfill
  \begin{subfigure}[t]{0.50\textwidth}
    \centering
    $\begin{hexboard}
      \rotation{-30}
      \hex(4,3)\black(4,3)\label{$1$}\hex(5,3)\white(5,3)\hex(6,3)\white(6,3)
      \hex(3,4)\white(3,4)\hex(4,4)\white(4,4)\label{$x$}\hex(5,4)\black(5,4)\label{$y$}\hex(6,4)\black(6,4)\label{$3$}
      \hex(3,5)\black(3,5)\label{$2$}\hex(4,5)\white(4,5)\hex(5,5)\white(5,5)
    \end{hexboard}$
    \caption*{$\whitestone*{$x$}\blackstone*{$y$}$}
  \end{subfigure}
  \hfill
  \begin{subfigure}[t]{0.50\textwidth}
    \centering
    $\begin{hexboard}
      \rotation{-30}
      \hex(4,3)\black(4,3)\label{$1$}\hex(5,3)\white(5,3)\hex(6,3)\white(6,3)      \hex(3,4)\white(3,4)\hex(4,4)\white(4,4)\label{$x$}\hex(5,4)\white(5,4)\label{$y$}\hex(6,4)\black(6,4)\label{$3$}
      \hex(3,5)\black(3,5)\label{$2$}\hex(4,5)\white(4,5)\hex(5,5)\white(5,5)
    \end{hexboard}$
    \caption*{$\whitestone*{$x$}\whitestone*{$y$}$}
  \end{subfigure}

  \caption[Completions of a region with two cells.]{The four completions of the region in \cref{fig:region}.}
  \label{fig:completions}
\end{figure}

\begin{definition}[Order on completions of a region]
  Let $R$ be a region of a $\Rex$ board, and let $a$ and $b$ be two completions of $R$.
  We say that $a\leq b$ as completions of $R$ if, for all larger boards $X$ to embed $R$ into and for all completions $x$ of $X$, if Left wins in $x(a)$, then Left wins in $x(b)$. We say that $a\inteq b$, or $a$ \emph{is equivalent to} $b$, when $a\leq b$ and $b\leq a$.
\end{definition}

The order on completions of a region is a preorder, i.e., it is transitive and reflexive but it is not always anti-symmetric (for example, $\blackstone*{$x$}\whitestone*{$y$} \inteq \whitestone*{$x$}\whitestone*{$y$}$ in \cref{fig:completions}).
To make this into a poset, we will look at the completions up to equivalence.
We call the poset formed by the completions modulo equivalence the \emph{outcome poset}.

Let us finish comparing the completions of the region from \cref{fig:region}.
The completion $\blackstone*{$x$}\whitestone*{$y$}$ connects \blackstone*{$1$} and \blackstone*{$2$} while $\whitestone*{$x$}\blackstone*{$y$}$ does not.
Because of this, if $R$ is plugged in to a larger board $X$ with a completion $x$ of $X$ such that \blackstone*{$1$} is connected to the northern terminal and \blackstone*{$2$} is connected to southern terminal, as in \cref{fig:01comp10}, then Left wins in the game $x(\whitestone*{$x$}\blackstone*{$y$})$ and loses in the game $x(\blackstone*{$x$}\whitestone*{$y$})$.
This means that $\whitestone*{$x$}\blackstone*{$y$}\nleq \blackstone*{$x$}\whitestone*{$y$}$.
With similar reasoning it can be shown that $\blackstone*{$x$}\whitestone*{$y$}$ is equivalent to $\whitestone*{$x$}\whitestone*{$y$}$. Thus, there are only three outcomes of $R$ up to equivalence, and they form the poset $\blackstone*{$x$}\blackstone*{$y$}\leq\blackstone*{$x$}\whitestone*{$y$}\leq\whitestone*{$x$}\whitestone*{$y$}$.

\begin{figure}[ht]
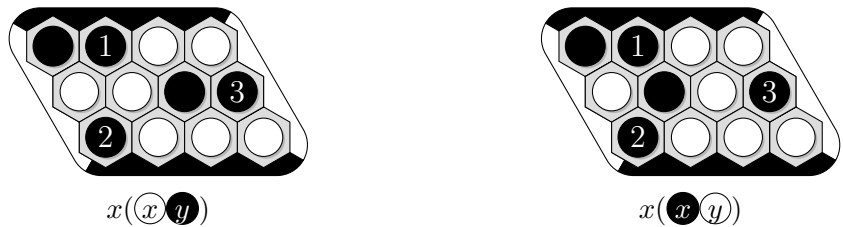

  \begin{subfigure}{0.5\textwidth}
  \centering
    $\begin{hexboard}
      \rotation{-30}
      \board(4,3) \black(1,1) \white (4,3)
      \hex(2,1)\black(2,1)\label{$1$}\hex(3,1)\white(3,1)\hex(4,1)\white(4,1)
      \hex(1,2)\white(1,2)\hex(2,2)\white(2,2)\hex(3,2)\black(3,2)\hex(4,2)\black(4,2)\label{$3$}
      \hex(1,3)\black(1,3)\label{$2$}\hex(2,3)\white(2,3)\hex(3,3)\white(3,3)
    \end{hexboard}$
    \caption*{$x(\whitestone*{$x$}\blackstone*{$y$})$}
  \end{subfigure}
  \hfill
  \begin{subfigure}{0.5\textwidth}
  \centering
    $\begin{hexboard}
      \rotation{-30}
      \board(4,3) \black(1,1) \white (4,3)
      \hex(2,1)\black(2,1)\label{$1$}\hex(3,1)\white(3,1)\hex(4,1)\white(4,1)
      \hex(1,2)\white(1,2)\hex(2,2)\black(2,2)\hex(3,2)\white(3,2)\hex(4,2)\black(4,2)\label{$3$}
      \hex(1,3)\black(1,3)\label{$2$}\hex(2,3)\white(2,3)\hex(3,3)\white(3,3)
    \end{hexboard}$
    \caption*{$x(\blackstone*{$x$}\whitestone*{$y$})$}
  \end{subfigure}
  
  \caption{Left wins with \whitestone*{$x$}\blackstone*{$y$} and loses with \blackstone*{$x$}\whitestone*{$y$}.}
  \label{fig:01comp10}
\end{figure}

\begin{figure}[ht]
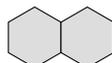

  \centering

  \begin{hexboard}
    \rotation{-30}
    \hex(0,0)\hex(1,0)
  \end{hexboard}
  
  \caption{Two adjacent empty cells.}
  \label{fig:two_empty_cells}
\end{figure}

Not all regions are bounded by stones like the region in \cref{fig:region}.
For example, \cref{fig:two_empty_cells}  shows a region of a Reverse Hex board with two empty cells and no stones along its border.
Again there are $2^2=4$ ways this region could be filled.
However, $\whitestone*{$x$}\blackstone*{$y$}$ and $\whitestone*{$x$}\whitestone*{$y$}$ are \emph{not} equivalent in this region, so the outcome poset is the four element poset shown in \cref{fig:two_empty_poset}.
\begin{figure}[ht]
  \centering
  \[
    \begin{tikzpicture}[scale=0.5]
      \node [label=$\top$] (top) at (0,1) {$\begin{hexboard}[scale=0.7]\rotation{-30}\hex(0,0)\white(0,0)\hex(1,0)\white(1,0)\end{hexboard}$};
      \node [label=0:$x$] (x) at (3,0) {$\begin{hexboard}[scale=0.5]\rotation{-30}\hex(0,0)\white(0,0)\hex(1,0)\black(1,0)\end{hexboard}$};
      \node [label=180:$y$] (y) at (-3,0) {$\begin{hexboard}[scale=0.5]\rotation{-30}\hex(0,0)\black(0,0)\hex(1,0)\white(1,0)\end{hexboard}$};
      \node [label=below:$\bot$] (bot) at (0,-1) {$\begin{hexboard}[scale=0.5]\rotation{-30}\hex(0,0)\black(0,0)\hex(1,0)\black(1,0)\end{hexboard}$};
      \path (top) edge [-] (x);
      \path (top) edge [-] (y);
      \path (x) edge [-] (bot);
      \path (y) edge [-] (bot);
    \end{tikzpicture}
  \]  
  \caption[Poset formed by two empty cells.]{The poset formed by two empty cells, ordered by preference.}
  \label{fig:two_empty_poset}
\end{figure}
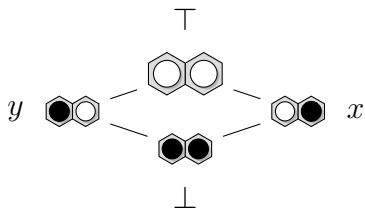

As mentioned on the previous page, given a region of a $\Rex$ board, the poset formed by its completions modulo equivalence is called the \emph{outcome poset} of the region.
For a game $G$ with a outcome poset $A$, we say that $G$ is played over $A$.

As discussed in the section about posets, we will use $\top$ to refer to the best outcome for Left and $\bot$ to refer to the worst outcome for Left.
Note that filling every empty cell with white stones is certainly the best completion for Left, and filling every empty cell with black stones is certainly the worst for Left.

The simplest example of an outcome poset is that of the region with one empty cell and no boundary conditions.
The two completions of this region naturally form the boolean poset $\bool$ (see \cref{fig:bool_poset}). 

\begin{figure}
  \[
    \begin{tikzpicture}
      \node [label=$\top$] (top) at (0,1) {$\begin{hexboard}\rotation{-30}\hex(0,0)\white(0,0)\end{hexboard}$};
      \node [label=below:$\bot$] (bot) at (0,-1) {$\begin{hexboard}\rotation{-30}\hex(0,0)\black(0,0)\end{hexboard}$};
      \path (top) edge [-] (bot);
    \end{tikzpicture}
  \]
  \caption{The boolean poset $\mathbb{B}$.}
  \label{fig:bool_poset}
\end{figure}
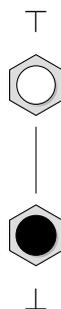

We now have the tools necessary to find the outcome poset of the region $G_2$ from \cref{fig:split_big_rex}:
\[\begin{hexboard}[scale=0.80, baseline=(current bounding box.center)]
    \rotation{-30}
    \foreach \i in {5,...,8}{ \hex(\i,1) }
    \white(5,1)\white(6,1)\white(7,1)\white(8,1)
    \foreach \i in {5,...,8}{ \hex(\i,2) }
    \black(5,2)\black(6,2)\white(7,2)\black(8,2)
    \foreach \i in {5,...,8}{ \hex(\i,3) }
    \white(5,3)\white(8,3)
    \foreach \i in {4,...,7}{ \hex(\i,4) }
    \white(4,4)\white(5,4)\white(7,4)
    \foreach \i in {3,...,6}{ \hex(\i,5) }
    \white(3,5)\white(4,5)\black(5,5)\black(6,5)
    \edge[\se\noobtusecorner\noacutecorner](3,5)(6,5)
    \node at \coord(4,6.5) {$G_2$};
  \end{hexboard} \]
This is an example of a 3-terminal region, i.e., a region with three black terminals and three white terminals.
There are at most five completions of a 3-terminal region up to equivalence.
The outcome $\top$ represents no two black terminals being connected, and the outcome $\bot$ represents when all three black terminals are connected by black stones.
The other outcomes arise from the three ways Left can connect two of the three terminals together. The outcome poset of $G_2$ is shown in \cref{fig:poset_G2}, along with an example of each outcome.

\begin{figure}
  \centering
  \begin{tikzpicture}[baseline=(current bounding box.center)]
    \node [label=$\top$] (top) at (0,2) {$\begin{hexboard}[scale=0.40]
        \rotation{-30}
        \foreach \i in {5,...,8}{ \hex(\i,1) }
        \white(5,1)\white(6,1)\white(7,1)\white(8,1)
        \foreach \i in {5,...,8}{ \hex(\i,2) }
        \black(5,2)\black(6,2)\white(7,2)\black(8,2)
        \foreach \i in {5,...,8}{ \hex(\i,3) }
        \white(5,3)\white(8,3)   \white(6,3)\white(7,3)
        \foreach \i in {4,...,7}{ \hex(\i,4) }
        \white(4,4)\white(5,4)\white(7,4)   \black(6,4)
        \foreach \i in {3,...,6}{ \hex(\i,5) }
        \white(3,5)\white(4,5)\black(5,5)\black(6,5)
        \edge[\se\noobtusecorner\noacutecorner](3,5)(6,5)
      \end{hexboard}$};
    \node [label=180:$a$] (a) at (-2.5,0) {$\begin{hexboard}[scale=0.40]
        \rotation{-30}
        \foreach \i in {5,...,8}{ \hex(\i,1) }
        \white(5,1)\white(6,1)\white(7,1)\white(8,1)
        \foreach \i in {5,...,8}{ \hex(\i,2) }
        \black(5,2)\black(6,2)\white(7,2)\black(8,2)
        \foreach \i in {5,...,8}{ \hex(\i,3) }
        \white(5,3)\white(8,3)   \black(6,3)\white(7,3)
        \foreach \i in {4,...,7}{ \hex(\i,4) }
        \white(4,4)\white(5,4)\white(7,4)   \black(6,4)
        \foreach \i in {3,...,6}{ \hex(\i,5) }
        \white(3,5)\white(4,5)\black(5,5)\black(6,5)
        \edge[\se\noobtusecorner\noacutecorner](3,5)(6,5)
      \end{hexboard}$};

    \node [label=0:$b$] (b) at (0,0) {$\begin{hexboard}[scale=0.40]
        \rotation{-30}
        \foreach \i in {5,...,8}{ \hex(\i,1) }
        \white(5,1)\white(6,1)\white(7,1)\white(8,1)
        \foreach \i in {5,...,8}{ \hex(\i,2) }
        \black(5,2)\black(6,2)\white(7,2)\black(8,2)
        \foreach \i in {5,...,8}{ \hex(\i,3) }
        \white(5,3)\white(8,3)   \black(6,3)\black(7,3)
        \foreach \i in {4,...,7}{ \hex(\i,4) }
        \white(4,4)\white(5,4)\white(7,4)   \white(6,4)
        \foreach \i in {3,...,6}{ \hex(\i,5) }
        \white(3,5)\white(4,5)\black(5,5)\black(6,5)
        \edge[\se\noobtusecorner\noacutecorner](3,5)(6,5)
      \end{hexboard}$};

    \node [label=0:$c$] (c) at (2.5,0) {$\begin{hexboard}[scale=0.40]
        \rotation{-30}
        \foreach \i in {5,...,8}{ \hex(\i,1) }
        \white(5,1)\white(6,1)\white(7,1)\white(8,1)
        \foreach \i in {5,...,8}{ \hex(\i,2) }
        \black(5,2)\black(6,2)\white(7,2)\black(8,2)
        \foreach \i in {5,...,8}{ \hex(\i,3) }
        \white(5,3)\white(8,3)   \white(6,3)\black(7,3)
        \foreach \i in {4,...,7}{ \hex(\i,4) }
        \white(4,4)\white(5,4)\white(7,4)   \black(6,4)
        \foreach \i in {3,...,6}{ \hex(\i,5) }
        \white(3,5)\white(4,5)\black(5,5)\black(6,5)
        \edge[\se\noobtusecorner\noacutecorner](3,5)(6,5)
      \end{hexboard}$};

    \node [label=below:$\bot$] (bot) at (0,-2) {$\begin{hexboard}[scale=0.40]
        \noshadows
        \rotation{-30}
        \foreach \i in {5,...,8}{ \hex(\i,1) }
        \white(5,1)\white(6,1)\white(7,1)\white(8,1)
        \foreach \i in {5,...,8}{ \hex(\i,2) }
        \black(5,2)\black(6,2)\white(7,2)\black(8,2)
        \foreach \i in {5,...,8}{ \hex(\i,3) }
        \white(5,3)\white(8,3)   \black(6,3)\black(7,3)
        \foreach \i in {4,...,7}{ \hex(\i,4) }
        \white(4,4)\white(5,4)\white(7,4)   \black(6,4)
        \foreach \i in {3,...,6}{ \hex(\i,5) }
        \white(3,5)\white(4,5)\black(5,5)\black(6,5)
        \edge[\se\noobtusecorner\noacutecorner](3,5)(6,5)
      \end{hexboard}$};
    \path (top) edge [-] (a)
    edge [-] (b)
    edge [-] (c);
    \path (bot) edge [-] (a)
    edge [-] (b)
    edge [-] (c);
  \end{tikzpicture}
  \caption[Another three terminal region.]{The outcome poset of a 3-terminal region}
  \label{fig:poset_G2}
\end{figure}
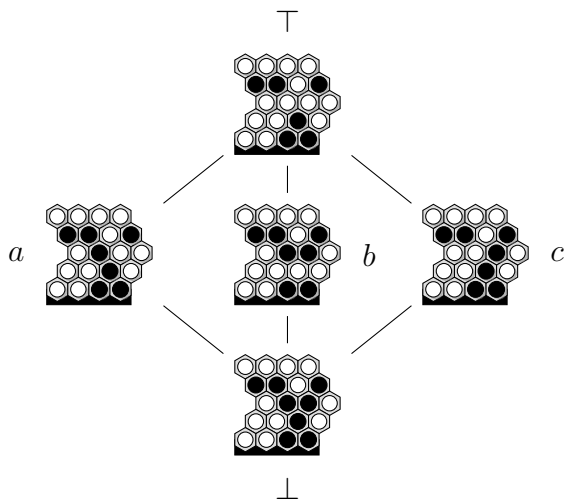
\section{Positions and their game forms}\label{sec:game_forms}

Now that we have an understanding of the possible completions of a region at the end of play, we need a treatment of a region in the midst of play.
Given a region, a \emph{position} of the region is any way one can assign the available cells to have a black stone, a white stone, or to be empty.
With three possible states for each available cell in a region, that means a region with $n$ available cells has $3^n$ positions.
For example, the nine positions in the region from \cref{fig:region} are shown in \cref{fig:positions}. All positions also come with a \emph{game form}, a description of what moves either player can make starting from that position. 
A position is called \emph{atomic} if it contains no empty cells, i.e., if it is a completion for the region.
\begin{definition}[Game form of a position]\label{def:game_form}
  Let $R$ be a region, and let $P$ be some position in the region $R$.
  If $P$ is not atomic, then the \emph{game form} $G=\gc{L}{R}$ of $P$ is a pair of sets $L$ and $R$, which are the sets of the game forms of the positions that Left/Right respectively can reach in one move from position $P$.
  If the position $P$ \emph{is} atomic, then the game form $G$ of $P$ is just the element of the outcome poset of $R$ which corresponds to that completion.
\end{definition}

\begin{figure}
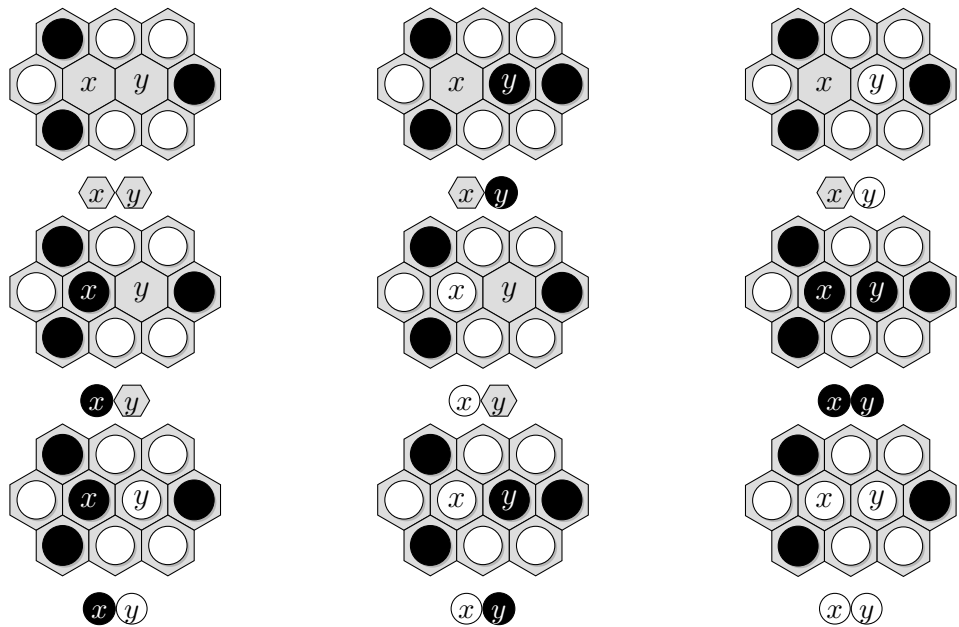

  \begin{subfigure}{0.3\textwidth}
    \centering
    $\begin{hexboard}
      \rotation{-30}
      \hex(4,3)\black(4,3)\hex(5,3)\white(5,3)\hex(6,3)\white(6,3)
      \hex(3,4)\white(3,4)\hex(4,4)\label{$x$}\hex(5,4)\label{$y$}\hex(6,4)\black(6,4)
      \hex(3,5)\black(3,5)\hex(4,5)\white(4,5)\hex(5,5)\white(5,5)
    \end{hexboard}$
    \caption*{\inlinehex*{$x$}\inlinehex*{$y$}}
  \end{subfigure}
  \hfill
  \begin{subfigure}{0.3\textwidth}
    \centering
    $\begin{hexboard}
      \rotation{-30}
      \hex(4,3)\black(4,3)\hex(5,3)\white(5,3)\hex(6,3)\white(6,3)
      \hex(3,4)\white(3,4)\hex(4,4)\label{$x$}\hex(5,4)\black(5,4)\label{$y$}\hex(6,4)\black(6,4)
      \hex(3,5)\black(3,5)\hex(4,5)\white(4,5)\hex(5,5)\white(5,5)
    \end{hexboard}$
    \caption*{\inlinehex*{$x$}\blackstone*{$y$}}
  \end{subfigure}
  \hfill
  \begin{subfigure}{0.3\textwidth}
    \centering
    $\begin{hexboard}
      \rotation{-30}
      \hex(4,3)\black(4,3)\hex(5,3)\white(5,3)\hex(6,3)\white(6,3)
      \hex(3,4)\white(3,4)\hex(4,4)\label{$x$}\hex(5,4)\white(5,4)\label{$y$}\hex(6,4)\black(6,4)
      \hex(3,5)\black(3,5)\hex(4,5)\white(4,5)\hex(5,5)\white(5,5)
    \end{hexboard}$
    \caption*{\inlinehex*{$x$}\whitestone*{$y$}}
  \end{subfigure}
  \hfill
  \begin{subfigure}{0.3\textwidth}
    \centering
    $\begin{hexboard}
      \rotation{-30}
      \hex(4,3)\black(4,3)\hex(5,3)\white(5,3)\hex(6,3)\white(6,3)
      \hex(3,4)\white(3,4)\hex(4,4)\black(4,4)\label{$x$}\hex(5,4)\label{$y$}\hex(6,4)\black(6,4)
      \hex(3,5)\black(3,5)\hex(4,5)\white(4,5)\hex(5,5)\white(5,5)
    \end{hexboard}$
    \caption*{\blackstone*{$x$}\inlinehex*{$y$}}
  \end{subfigure}
  \hfill
  \begin{subfigure}{0.3\textwidth}
    \centering
    $\begin{hexboard}
      \rotation{-30}
      \hex(4,3)\black(4,3)\hex(5,3)\white(5,3)\hex(6,3)\white(6,3)
      \hex(3,4)\white(3,4)\hex(4,4)\white(4,4)\label{$x$}\hex(5,4)\label{$y$}\hex(6,4)\black(6,4)
      \hex(3,5)\black(3,5)\hex(4,5)\white(4,5)\hex(5,5)\white(5,5)
    \end{hexboard}$
    \caption*{\whitestone*{$x$}\inlinehex*{$y$}}
  \end{subfigure}
  \hfill
  \begin{subfigure}{0.3\textwidth}
    \centering
    $\begin{hexboard}
      \rotation{-30}
      \hex(4,3)\black(4,3)\hex(5,3)\white(5,3)\hex(6,3)\white(6,3)
      \hex(3,4)\white(3,4)\hex(4,4)\black(4,4)\label{$x$}\hex(5,4)\black(5,4)\label{$y$}\hex(6,4)\black(6,4)
      \hex(3,5)\black(3,5)\hex(4,5)\white(4,5)\hex(5,5)\white(5,5)
    \end{hexboard}$
    \caption*{\blackstone*{$x$}\blackstone*{$y$}}
  \end{subfigure}
  \hfill
  \begin{subfigure}{0.3\textwidth}
    \centering
    $\begin{hexboard}
      \rotation{-30}
      \hex(4,3)\black(4,3)\hex(5,3)\white(5,3)\hex(6,3)\white(6,3)
      \hex(3,4)\white(3,4)\hex(4,4)\black(4,4)\label{$x$}\hex(5,4)\white(5,4)\label{$y$}\hex(6,4)\black(6,4)
      \hex(3,5)\black(3,5)\hex(4,5)\white(4,5)\hex(5,5)\white(5,5)
    \end{hexboard}$
    \caption*{\blackstone*{$x$}\whitestone*{$y$}}
  \end{subfigure}
  \hfill
  \begin{subfigure}{0.3\textwidth}
    \centering
    $\begin{hexboard}
      \rotation{-30}
      \hex(4,3)\black(4,3)\hex(5,3)\white(5,3)\hex(6,3)\white(6,3)
      \hex(3,4)\white(3,4)\hex(4,4)\white(4,4)\label{$x$}\hex(5,4)\black(5,4)\label{$y$}\hex(6,4)\black(6,4)
      \hex(3,5)\black(3,5)\hex(4,5)\white(4,5)\hex(5,5)\white(5,5)
    \end{hexboard}$
    \caption*{\whitestone*{$x$}\blackstone*{$y$}}
  \end{subfigure}
  \hfill
  \begin{subfigure}{0.3\textwidth}
    \centering
    $\begin{hexboard}
      \rotation{-30}
      \hex(4,3)\black(4,3)\hex(5,3)\white(5,3)\hex(6,3)\white(6,3)
      \hex(3,4)\white(3,4)\hex(4,4)\white(4,4)\label{$x$}\hex(5,4)\white(5,4)\label{$y$}\hex(6,4)\black(6,4)
      \hex(3,5)\black(3,5)\hex(4,5)\white(4,5)\hex(5,5)\white(5,5)
    \end{hexboard}$
    \caption*{\whitestone*{$x$}\whitestone*{$y$}}
  \end{subfigure}
  \caption[Positions of the region in \cref{fig:region}.]{The nine positions of the region in \cref{fig:region}.}
  \label{fig:positions}
\end{figure}

This recursive definition is similar to the definition for normal play games.
As an example, consider the region $R$ from \cref{fig:region}.
Let us write $\inlinehex*{$x$}\inlinehex*{$y$}$ when both cells are empty, and so on.

From the position \inlinehex*{$x$}\inlinehex*{$y$} Left could move to \inlinehex*{$x$}\blackstone*{$y$} or \blackstone*{$x$}\inlinehex*{$y$} and Right could move to \whitestone*{$x$}\inlinehex*{$y$} or \inlinehex*{$x$}\whitestone*{$y$}.
The game form of \inlinehex*{$x$}\blackstone*{$y$} is $\gc{\blackstone*{$x$}\blackstone*{$y$}}{\whitestone*{$x$}\blackstone*{$y$}}$ because $\blackstone*{$x$}\blackstone*{$y$}$ and $\whitestone*{$x$}\blackstone*{$y$}$ are both completions of $R$.
The game form of \blackstone*{$x$}\inlinehex*{$y$} is $\gc{\blackstone*{$x$}\blackstone*{$y$}}{\blackstone*{$x$}\whitestone*{$y$}}$.
For \whitestone*{$x$}\inlinehex*{$y$} and \inlinehex*{$x$}\whitestone*{$y$} the game forms are $\gc{\whitestone*{$x$}\blackstone*{$y$}}{\whitestone*{$x$}\whitestone*{$y$}}$ and $\gc{\blackstone*{$x$}\whitestone*{$y$}}{\whitestone*{$x$}\whitestone*{$y$}}$ respectively.
Then the game form of the position \inlinehex*{$x$}\inlinehex*{$y$} is \[\gc{\strut\gc{\blackstone*{$x$}\blackstone*{$y$}}{\whitestone*{$x$}\blackstone*{$y$}}, \,\gc{\blackstone*{$x$}\blackstone*{$y$}}{\blackstone*{$x$}\whitestone*{$y$}} }{\gc{\whitestone*{$x$}\blackstone*{$y$}}{\whitestone*{$x$}\whitestone*{$y$}},\, \gc{\blackstone*{$x$}\whitestone*{$y$}}{\whitestone*{$x$}\whitestone*{$y$}}}.\]

\section{Games over posets}

For the rest of this thesis we will focus on game forms over arbitrary posets.
We define the class of game forms over some poset as follows:
\begin{definition}[Game forms over a poset]\label{def:games_posets}
  Fix a poset $A$, called the \emph{outcome poset}. The class of \emph{game forms over $A$} is the smallest class satisfying:
  \begin{enumerate}
  \item for all $a\in A$, $[a]$ is a game form.
  \item Whenever $L$ and $R$ are non-empty sets of game forms, then $\gc{L}{R}$ is a game form.
  \end{enumerate}
\end{definition}

We will use the usual terminology and notations of combinatorial game theory.
When no confusion arises, we simply refer to a game form as a \emph{game}.
A game of the form $[a]$ is called an \emph{atomic game}.
We often just write $a$ instead of $[a]$: the notation $[a]$ is used for operations on games that translate to operations on the poset.
A game of the form $\gc{L}{R}$ is called a \emph{composite game}, and the elements of $L$ and $R$ are called its \emph{left} and \emph{right} options, respectively.
We will often write $G=\gc{G^L}{G^R}$ to represent a game, where $G^L$ is a typical left option of $G$ and $G^R$ is a typical right option of $G$.
Then we use ``for all $G^L$'' to mean ``for all left options $G^L$ of $G$,'' and similarly for right options.
When a game is known to be atomic, we often write it in lowercase.
We sometimes use $G:A$ to denote that $G$ is a game over the poset $A$.

For example, let $\bot\leq h\leq\top$ be a poset.
This poset is isomorphic to the outcome poset of the region from \cref{fig:region}, where $\blackstone*{$x$}\blackstone*{$y$}$ maps to $\bot$, $\blackstone*{$x$}\whitestone*{$y$}$ maps to $h$, and $\whitestone*{$x$}\blackstone*{$y$}$ and $\whitestone*{$x$}\whitestone*{$y$}$ both map to $\top$.
Then
\[\gc{\gc{\bot}{\top}, \,\gc{\bot}{h} }{\gc{\top}{\top},\, \gc{h}{\top}}\] is another representation of the game form of the empty position in this region.

There are game forms that \emph{cannot} be realized as a position of some $\Rex$ region.
For example, let $G$ be the game form $G=\gc{\top}{\bot}$ over some poset $A$.
If we wanted to find a $\Rex$ position with this game form, we would need a position with one empty cell remaining such that filling the cell with a black stone is winning for Left and filling the cell with a white stone is winning for Right. This is impossible because $\Rex$ is antimonotone (see \cref{setColoringGame}).

\begin{definition}[Sums of games]\label{sum_on_game}
  Let $G$ be a game over the poset $A$, and let $H$ be a game over the poset $B$.
  The game $G+H$ is over the poset $A\times B$ and is defined as follows:
  \begin{itemize}
  \item If $[g]$ and $[h]$ are atomic, then $[g]+[h]=[(g,h)]$.
  \item If $G$ is composite and $h$ is atomic, then $G+h=\gc{G^L+h}{G^R+h}$.
  \item If $g$ is atomic and $H$ is composite, then $g+H=\gc{g+H^L}{g+H^R}$.
  \item Otherwise, $G+H=\gc{G^L+H,\,G+H^L}{G^R+H,\,G+H^R}$.
  \end{itemize}
\end{definition}
Note that, when both games are composite, this definition agrees with the definition of the disjunctive sum on normal play games.
This sum is also associative: $(G+(H+K))=((G+H)+K)$.

\begin{definition}[Map operation] \label{map_on_game}
  Let $\phi\colon A\to B$ be a monotone function and let $G$ be a game over $A$.
  Then $\phi(G)$ is a game over $B$, defined as follows:
  \begin{itemize}
  \item If $[g]$ is atomic, then $\phi([g])=[\phi(g)]$
  \item If $G$ is composite, then $\phi(G)=\gc{\phi(G^L)}{\phi(G^R)}$.
  \end{itemize}
\end{definition}

Let $f\colon A\times B\to C$ be a monotone function where $A$, $B$, and $C$ are posets. Often in this thesis we will apply $f$ to the sum $G+H$, with $G$ and $H$ being games over $A$ and $B$, respectively. We write $G+_f H$ as shorthand for $f(G+H)$.

Similarly to normal play games the \emph{dual} of a game $G$, denoted by $\op{G}$, refers to the game $G$ where Left and Right have swapped places.
Right wants what Left does not: if Left believes outcome $a$ is worse than outcome $b$, then Right \emph{prefers} outcome $a$ over outcome $b$.
Because of this, if $G$ is over some poset $A$ then $\op{G}$ is over $\op{A}$.
We define this formally as follows:
\begin{definition}[Dual game]\label{def:dual2}
  Let $G$ be a game over a poset $A$.
  If $G$ is composite, then the \emph{dual} of $G$ is $\op{G}=\gc{\op{(G^R)}}{\op{(G^L)}}$.
  If $[g]$ is atomic, then $\op{[g]}=[\op{g}]$, where $\op{g}$ is the corresponding element of $\op{A}$.
\end{definition}
We often write $\op{G^L}$ and $\op{G^R}$ without parenthesis.
Note that, when $G$ is composite, the definition of $\op{G}$ aligns with the normal play definition.

\begin{definition}[First and second player winning strategy for a game over $\bool$]
  Let $G$ be a game over $\bool$.
  \begin{itemize}
  \item Left has a \emph{first player winning strategy} in $G$ if there is some $G^L$ of $G$ such that Left has a second player strategy in $G^L$ \emph{or} $[g]$ is atomic and $g=\top$.
  \item Left has a \emph{second player winning strategy} in $G$ if Left has a first player winning strategy in $G^R$ for all $G^R$ of $G$ and, if $[g]$ is atomic, then $g=\top$.
  \end{itemize}
\end{definition}

\begin{theorem}[Induction on games over posets]
  Let $P$ be a property of games over some poset $A$.
  Suppose that $P(g)$ is true for all atomic games $g \in A$ and suppose that, given a composite game $G$, if $P(G')$ is true for all left options and right options $G'$ of $G$, then $P(G)$ is true.
  Then $P(G)$ is true for all games $G$ over $A$. \end{theorem}
\begin{proof}
  Let $P$ be a property as defined above and let $C$ be the class of all games that satisfy $P$.
  By assumption the game $[x]$ is in $C$ for all $x\in A$ and the game $G=\gc{L}{R}$ is in $C$ for any two non-empty sets of games $L$ and $R$ in $C$.
  We defined the class of games over the poset $A$ to be the \emph{smallest} class which satisfies the above conditions (see \cref{def:games_posets}).
  Then the class of games over $A$ is contained in $C$, which implies $P(G)$ is true for all $G$ over $A$.
\end{proof}

In \cref{rex_premotivity} we informally defined the property of \emph{premotivity}.
Now we define this formally.
Given two games $G$ and $H$ over the poset $A$, we refer to $\lcomp{G}{H}$ as the \emph{comparison game}, where $\lambda$ is the comparison function defined in \cref{Background:mono_fun}.
The game that Alice played against Charboblie can be written as $\lcomp{G}{G}$ , where Bob's board is $G$ and Charlie's board is $\op{G}$.
We can prove that the second player copycat strategy holds for all games over posets:
\begin{proposition}[Second player copycat strategy]\label{prop:copycat}
Let $G$ be a game over a poset $A$. Then Left has a second player winning strategy in the game $\lcomp{G}{G}$. \end{proposition}
\begin{proof}
  Assume $[g]$ is an atomic game. Then $\lcomp{g}{g}$ is equal to $\top$ because $A$ is reflexive.
  
  Now assume $G$ is composite.
  By the induction hypothesis, Left has a second player winning strategy in the game $\lcomp{(G')}{G'}$ for all left or right options $G'$ of $G$.
  If Right moves to $\lcomp{G}{G^R}$, then Left copies, moving to $\lcomp{(G^R)}{G^R}$.
  If Right moves to $\lcomp{(G^L)}{G}$, then Left copies, moving to $\lcomp{(G^L)}{G^L}$.
  So Left has a first player winning strategy in the games $\lcomp{G}{G^R}$ and $\lcomp{(G^L)}{G}$ for all $G^L$ and $G^R$ of $G$.
  Therefore Left has a second player winning strategy in the game $\lcomp{G}{G}$, as desired. \end{proof}

\begin{definition}[Premotivity]\label{premotive}
  We say that a game $G$ is \emph{premotive} if Left has a first player winning strategy in the game $\lcomp{G}{G}$.
\end{definition}

In \cref{Properties_of_Rex_revisited} we will prove that premotive games satisfy the lookahead property (see \cref{prop_x}).

\chapter{Evaluating games}\label{Evaluating_games}

In normal play games, we say that $G\leq H$ if Left would prefer $H$ over $G$ in any game that can be written as $G+X$, where $X$ is some other game.
Because the game $X$ can be thought of as a \emph{context} in which to play $G$ or $H$, we call this the \emph{contextual} definition of the $\leq$ relation.
As we saw in \cref{normal_intrinsic}, in the case of normal play games, there is also an equivalent \emph{intrinsic} definition of the $\leq$ relation, i.e., one that does not refer to any contexts.

For our games over posets, the situation is slightly more complicated.
We will define a contextual order in \cref{cont_order} and an intrinsic order in \cref{intrinsic_order}.
These orders will not be equivalent over the class of all games, but they will become equivalent when restricted to a special subclass of games.

\section{The contextual order}\label{cont_order}

We want to define something similar to the contextual order on games over a poset $A$.
Informally we say \emph{$G$ is less than or equal to $H$ contextually}, denoted by $G\leqc H$, if Left would prefer $H$ over $G$ regardless of the context, i.e., whatever is happening on the rest of the board.

Let $A$ be a poset and recall that $\fun{A}$ is the poset of monotone functions from $A$ to $\bool$ (see \cref{Background:Posets}).
We define $\varepsilon\colon A\times\fun{A}\to\bool$ to be the application function, where \[\varepsilon(a,f)=f(a).\]

When embedding a $\Rex$ position into the completion of a larger board, we can think of this completion as an element of $\fun{A}$. Because of this, we define a \emph{context} to be a game $X$ over $\fun{A}$.
Then the game $G+_\varepsilon X$ is the game $G$ embedded into the game $X$.
From here on, when we add a game to a context will write $G + X$, omitting the $\varepsilon$.

\begin{definition}[Contextual order on games]
  Let $G$ and $H$ be games over some poset $A$.
  \begin{itemize}
  \item We say $G\leqc H$ if $o(G + X)\leq o(H + X)$ for all contexts $X$.
  \item We say $G\contri H$ if $o_R(G+X) \leq o_L(H+X)$ for all contexts $X$.
  \end{itemize}
\end{definition}
We say that $G$ is \emph{contextually equivalent} to $H$, denoted by $G\conteq H$, if $G\leqc H$ and $H\leqc G$.
The contextual relation $\leqc$ is transitive and reflexive, so it forms a preorder on the set of games.
In fact, we also have that $G\leqc H\contri K$ implies $G\contri K$ and $G\contri H\leqc K$ implies $G\contri K$.
It is also clear that $G\leqc H$ implies $G+K\leqc H+K$ for any game $K$.

With this definition we can already prove some results about games that we might expect.
The following lemma is an example of this.
Informally, the lemma says that replacing a left option of a game with something better is always good for Left.
\begin{lemma}Suppose $H\leqc K$.
  
  If $G\identical\gc{G^L,H}{G^R}$ and $G'\identical\gc{G^L,K}{G^R}$, then $G\leqc G'$.
  \label{congruent_option_replacement}
\end{lemma}

\begin{proof}

  We will show this by proving the following two claims using induction on $X$:

  \begin{enumerate}[(A)]
  \item $o_L(G+X)\leq o_L(G'+X)$ and
  \item $o_R(G+X)\leq o_R(G'+X)$.
  \end{enumerate}
  We will prove claim (A) first.
  Assume Left has a first player winning strategy in the game $G+X$.
  Then there must be some $(G+X)^L$ such that Left has a second player winning strategy in $(G+X)^L$.
  We want to show that Left has a first player winning strategy in the game $G'+X$.
  If $(G+X)^L$ is of the form $G+X^L$ for some  $X^L$ of $X$, then Left can copy that move in $G'+X$.
  By the induction hypothesis, $o(G+X^L)\leq o(G'+X^L)$.
  Then Left has a second player winning strategy in $G'+X^L$.
  
  If $(G+X)^L$ is of the form $G^L+X$ for $G^L\neq H$, then Left can just move to $G^L+X$ in the game $G'+X$.
  Finally, assume $(G+X)^L$ is of the form $H+X$.
  Because $H\leqc K$, we know that $o(H+X)\leq o(K+X)$.
  Then Left moves to the game $K+X$ in the game $G'+X$.
  Then we have shown that $o_L(G+X)\leq o_L(G'+X)$.

  Now we will prove claim (B).
  Assume Left has a second player winning strategy in the game $G+X$.
  Then for all $(G+X)^R$ of $G+X$ Left has a first player winning strategy in $(G+X)^R$.
  We want to show that Left has a first player winning strategy in $(G'+X)^R$ for all $(G'+X)^R$ of $G'+X$.
  
  The Left's first player winning strategy in $G^R+X$ can be copied directly because the right options of $G$ and $G'$ are the same.
  Assume Right moves to some $G'+X^R$ of $G'+X$. By the induction hypothesis $o(G+X^R)\leq o(G'+X^R)$, so Left has a first player winning strategy in $G'+X^R$.
  Therefore, Left has a second player winning strategy in $G'+X$, as desired.
\end{proof}

\section{The intrinsic order}\label{intrinsic_order}

The contextual order is a valid order to impose on games over posets, but it is difficult to verify.
It is not reasonable to write a program that checks if $o(G+X)\leq o(G+X)$ for all contexts $X$ because there are infinitely many contexts $X$.
As is usual in CGT, we would like an order on games which is equivalent to the contextual order and does not refer to contexts.

\begin{definition}[Intrinsic order on games over a poset]
  Let $G$ and $H$ be games over the same poset $A$.
  We say $G\leqint H$ when \emph{all} of the following hold:
  \begin{enumerate}
  \item For all left options $G^L$ of $G$, we have $G^L\tri H$, and
  \item for all right options $H^R$ of $H$, we have $G\tri H^R$, and
  \item if $G\identical g$, $H\identical h$ are atomic games, then $g\leq_A h$.
  \end{enumerate}
  We say $G\tri H$ when \emph{at least one} of the following holds:
  \begin{enumerate}
  \item There exists some $G^R$ of $G$ such that $G^R\leqint H$, or
  \item there exists some $H^L$ of $H$ such that $G\leqint H^L$.
  \end{enumerate}
  We say that $G\inteq H$, or $G$ \emph{is intrinsically equivalent to} $H$, when $G\leq H$ and $H\leq G$.
\end{definition}

\begin{proposition}\label{comp_int}
  Given two games $G$ and $H$ over the same poset, we have the following:
  \begin{enumerate}[(A)]
  \item $G\leqint H$ if and only if Left has a second player winning strategy in the comparison game $\lcomp{G}{H}$ in which Left makes the last move.
  \item $G\tri H$ if and only if Left has a first player winning strategy in $\lcomp{G}{H}$ in which Left makes the last move.
  \end{enumerate}
\end{proposition}

\begin{proof}
  Let $G$ and $H$ be two games over some poset $A$. We will prove both claims by simultaneous induction on the games $G$ and $H$.

  First, we will prove claim (A).
  Assume $G\leqint H$.
  If $G=g$ and $H=h$ are both atomic, then by the definition of $\lambda$ Left has already won the game $\lcomp{g}{h}$.
  Now assume that $G$ and $H$ are not both atomic.
  We want to prove that Left has a second player winning strategy in the comparison game $\lcomp{G}{H}$ in which Left makes the last move.
  The options of $\lcomp{G}{H}$ which Right can move to are either of the form $\lcomp{G^L}{H}$ or of the form $\lcomp{G}{H^R}$. By the definition of $G\leqint H$, for all $G^L$ of $G$ and $H^R$ of $H$ we have $G^L\tri H$ and $G\tri H^R$, respectively.
  Using part (B) of the induction hypothesis, this implies that in every right option of the game $\lcomp{G}{H}$ there is a first player winning strategy for Left in which Left makes the last move.
  This means Left has a second player winning strategy in the game $\lcomp{G}{H}$ in which Left makes the last move.

  Next we assume that, in the game $\lcomp{G}{H}$, Left has a second player winning strategy in which they make the last move.
  We want to prove that $G\leqint H$.
  Let $G^L$ be some left option of $G$.
  The game $\lcomp{G^L}{H}$ is a right option of $\lcomp{G}{H}$, so Left must have a first player winning strategy in which Left moves last.
  Then using part (B) of the induction hypothesis, we have that $G^L\tri H$.
  By the same argument we see that $G\tri H^R$ for all right options $H^R$ of $H$.
  If the game $\lcomp{G}{H}$ is atomic, then $g$ and $h$ must be atoms and $g\leq h$.
  Then we have that $G\leqint H$ as desired.

  Now we will prove claim (B).
  Assume $G\tri H$.
  Either there exists some $G^R$ of $G$ such that $G^R\leqint H$ or there exists some $H^L$ of $H$ such that $G\leqint H^L$.
  Both $\lcomp{G^R}{H}$ and $\lcomp{G}{H^L}$ are left options of the game $\lcomp{G}{H}$.
  Then using part (A) of the induction hypothesis, there exists a left option of $\lcomp{G}{H}$ in which Left has a second player strategy to win and make the last move.
  Then in the game $\lcomp{G}{H}$ Left has a first player strategy to win and make the last move, as desired.

  Finally, assume that in the game $\lcomp{G}{H}$ Left has a first player strategy to win and make the last move.
  We want to show that $G\tri H$.
  Note that ``Left has the last move'' means that Left always has a move on their turn.
  Then $G$ and $H$ cannot both be atomic, because Left would have no move to make.
  Then there must be some left option of $\lcomp{G}{H}$ where Left has a second player strategy to win and make the last move.
  That left option will either be of the form $\lcomp{G^R}{H}$ or of the form $\lcomp{G}{H^L}$.
  Using part (A) of the induction hypothesis, we see that either there exists some $G^R$ of $G$ such that $G^R\leqint H$ or there exists some $H^L$ of $H$ such that $G\leqint H^L$.
  This means that $G\tri H$ as desired.
\end{proof}

\begin{cor}
  Let $G$ and $H$ be games over the same poset $A$. Then $G\leq H$ if and only if $\top\leq \lcomp{G}{H}$, and $G\tri H$ if and only if $\top\tri \lcomp{G}{H}$.
  \label{comparison_game_II}
\end{cor}
\begin{proof}
  As shown above, $G\leqint H$ if and only if Left has a second player strategy in the game $\lcomp{G}{H}$ to win and make the last move.
  That must be true if and only if Left has a second player strategy in the game $\lcomp{\top}{(\lcomp{G}{H})}$ to win and make the last move, which is true if and only if $\top\leqint \lcomp{G}{H}$, as desired.

  The proof that $G\tri H$ if and only if $\top\tri\lcomp{G}{H}$ is similar.
\end{proof}

Ideally, we would like to have that $G\leq H$ if and only if $G\leqc H$. Unfortunately, this is not always the case. As it stands, this intrinsic relation is not even transitive. For example, we have:

\begin{equation}
  \label{eq:not-trans}
  \top\leq\gc{\gc{\bot}{\bot}}{\gc{\top}{\top}}\leq\bot.
\end{equation}
We will spend the next few chapters looking at properties which the game form of any $\Rex$ position (and more generally, any set coloring game with antimonotone payoff function) has.
These are the properties of parity, premotivity, and $*$-antimonotonicity that we already gave examples of in \cref{Rex_concrete}.
It turns out that the class of premotive games with parity is well-behaved:
on it, the intrinsic order is transitive and coincides with the contextual order.
Also, this class of games admits canonical forms.

\section{Intrinsic order and sum}
The following proposition shows that the intrinsic order is preserved by summation.
\begin{proposition}\label{sum_preserves_order}
  For games $G$ and $G'$ over some poset $A$ and games $H$ and $H'$ over some poset $B$, all of the following hold:
  \begin{enumerate}[(A)]
  \item If $G\leqint G'$ and $H\leqint H'$, then $G+H\leqint G' + H'$.
  \item if $G\tri G'$ and $H\leqint H'$, then $G+H\tri G' + H'$.
  \item if $G\leqint G'$ and $H\tri H'$, then $G+H\tri G' + H'$.
  \end{enumerate}
\end{proposition}

\begin{proof}
  Let $G$ and $G'$ be games over the poset $A$ and let $H$ and $H'$ be games over the poset $B$.
  We will prove these claims using joint induction.

  First we will prove claim (A).
  Assume that $G\leqint G'$ and $H\leqint H'$.
  We want to show that $G+H\leqint G'+H'$.
  If $g$, $h$, $g'$, and $h'$ are all atomic, then  $(g,h)\leq_{A\times B} (g',h')$ as desired.
  Suppose at least one of the games is not atomic.
  To show $G+H\leqint G'+H'$, we will need to show that $(G+H)^L\tri G'+H'$ for all left options $(G+H)^L$ of $G+H$ and that $G+H\tri (G'+H')^R$ for all right options $(G'+H')^R$ of $G'+H'$.
  We start by proving the first part; take some left option of $G+H$, of the form $G^L+H$.
  Because $G\leqint G'$, we know that $G^L\tri G$ for all left options $G^L$ of $G$.
  By part (B) of the induction hypothesis, $G^L\tri G'$ and $H\leqint H'$ implies $G^L+H\tri G'+H'$.
  Now take a left option of the form $G+H^L$.
  Similarly, by using part (C) of the induction hypothesis, we have that $G\leqint G'$ and $H^L\tri H$ implies $G+H^L\tri G'+H'$.
  This proves that $(G+H)^L\leqint G'+H'$ for any left option $(G+H)^L$ of $G+H$, as desired.
  We omit the proof of the second claim because it is dual to the proof above.

  Next we will prove claim (B).
  Assume that $G\tri G'$ and $H\leqint H'$.
  Because $G\tri G'$, either there exists some $G^R$ of $G$ such that $G^R\leqint G'$ or there exists some $G'^L$ of $G'$ such that $G\leqint G'^L$.
  In the first case, $G^R\leqint G'$ and $H\leqint H'$ implies $G^R+H\leqint G'+H'$ by the induction hypothesis, which implies that $G+H\tri G'+H'$.
  In the second case, $G\leqint G'^L$ and $H\leqint H'$ implies that $G+H\leqint G'^L+H'$ by part (A) of the induction hypothesis, which implies that $G+H\tri G'+H'$.

  The proof of (C) is symmetric to the proof of (B).
\end{proof}


\section{Intrinsic order and monotone functions}
The next two propositions show that the intrinsic order behaves as expected with respect to monotone functions applied to games.
\begin{proposition}
  Let $G$ and $H$ be games over some poset $A$, and let $\phi$ and $\psi$ be monotone functions from the poset $A$ to some other poset $Z$.
  If $\phi\leq\psi$ and $G\leq H$, then $\phi(G)\leqint\psi(H)$.
  \label{fun_equiv}  
\end{proposition}
\begin{proof}
  Assume $G$ and $H$ are games over the poset $A$, and $\phi\colon A\rightarrow Z$ and $\psi\colon A\rightarrow Z$ are monotone functions such that $\phi\leq\psi$.
  We will prove the following claims by induction on $G$ and $H$:
  \begin{enumerate}[(A)]
  \item If $G\leqint H$, then $\phi(G)\leqint\psi(H)$.
  \item If $G\tri H$, then $\phi(G)\tri\psi(H)$.
  \end{enumerate}

  First we prove claim (A).
  If $g$ and $h$ are atomic, then $\phi(g)\leq\psi(h)$ by assumption.
  Assume $G$ or $H$ is composite.
  Let $G^L$ be some left option of $G$.
  We know that $G^L\tri H$, so by the induction hypothesis, $\phi(G^L)\leqint\psi(H)$.
  Let $H^R$ be some right option of $H$.
  By dual reasoning, $\psi(G)\tri\phi(H^R)$.
  Then $\phi(G)\leqint\psi(H)$ as desired.
  
  Now we prove claim (B).
  Assume $G\tri H$.
  Without loss of generality, assume there exists some $H^L$ of $H$ such that $G\leqint H^L$.
  Then by the induction hypothesis, $\phi(G)\leqint\psi(H^L)$.
  Then we have proved that $\phi(G)\tri\psi(H)$. \end{proof}

\begin{proposition}\label{fun_equiv_trans}
  Let $G$ be a game over some poset $A$, and let $\phi$ and $\psi$ be monotone functions from the poset $A$ to some other poset $Z$. Suppose $\phi\leq\psi$. Then we have the following:
  \begin{enumerate}[(A)]
  \item If $\top\leqint\phi(G)$, then $\top\leqint\psi(G)$.
  \item If $\top\tri\phi(G)$, then $\top\tri\psi(G)$.
  \end{enumerate}  \end{proposition}

\begin{proof}
  Assume $G$ is a game over $A$ and $\phi$ and $\psi$ are defined as above.
  We will prove both claims using induction on $G$.

  First we will prove claim (A).
  If $g$ is atomic, then $\top\leq\psi(g)$ by the transitivity of posets.
  Assume $G$ is composite.
  The game $\top$ has no options, so we only need to show that $\top\tri\psi(G^R)$ for all $G^R$ of $G$.
  Because $\top\leqint\phi(G)$, we know that $\top\tri\phi(G^R)$.
  This means that $\top\tri\psi(G^R)$ by using part (B) of the induction hypothesis.
  This is true for any $G^R$ of $G$, therefore $\top\leqint\psi(G)$.

  Now we prove claim (B).
  Assume $\top\tri\phi(G)$.
  Because the game $\top$ has no options, there must exist some  $G^L$ of $G$ such that $\top\leqint\phi(G^L)$.
  Using part (A) of the induction hypothesis, this implies that $\top\leqint\psi(G^L)$ for that same $G^L$.
  Therefore, $\top\tri\psi(G)$.
\end{proof}
\chapter{Properties of \texorpdfstring{$\Rex$}{Rex} revisited} \label{Properties_of_Rex_revisited}
In this chapter, we will demonstrate how to extend some of the special properties of $\Rex$ (as discussed in \cref{Rex_concrete}) to arbitrary games over a poset.

\section{Parity}
\begin{definition}[Parity of a game]
  We say that a game is \emph{even} if all of its options are odd. We say that a game is \emph{odd} if it is composite and all of its options are even. A game has \emph{parity} if it is even or odd.
\end{definition}
Note that some games, such as $\gc{\bot}{\gc{\top}{\top}}$, are neither even nor odd.
\begin{lemma}\label{parity_trans}
  Let $G$, $H$, and $K$ be games over some poset $A$, and suppose $G$ and $K$ have parity. Then:
  \begin{enumerate}[(A)]
  \item If $G\leq H\leq K$, then $G$ and $K$ have matching parity.
  \item If $G\leq H\tri K$, then $G$ and $K$ have mismatching parity.
  \item If $G\tri H\leq K$, then $G$ and $K$ have mismatching parity.
  \end{enumerate}
\end{lemma}
\begin{proof}
  Let $G$, $H$, and $K$ be as described above.
  We will prove all four claims using induction on $G$, $H$, and $K$.

  First we prove claim (A). Suppose $G\leq H\leq K$. If $G=[g]$ and $K=[k]$ are both atomic, then they automatically have matching parity. Then assume $G$ or $K$ is composite. If $G$ is composite, then we have $G^L\tri H\leq K$. By part (B) of the induction hypothesis we know $G^L$ and $K$ have mismatching parity. We know $G$ and $G^L$ have mismatching parity, so $G$ and $K$ have matching parity, as desired. Similar reasoning applies when $K$ is composite.

  Now we will prove claim (B). Suppose that $G\leq H \tri K$.
  Because $H\tri K$, either there exists some $H^R$ of $H$ such that $H^R\leq K$ or there exists some $K^L$ of $K$ such that $H\leq K^L$.

  If there is some $H^R$ of $H$ such that $H^R\leq K$, then $G\tri H^R\leq K$.
  Then, by part (C) of the induction hypothesis, we know that $G$ and $K$ have mismatching parity.
  If instead there is some $K^L$ such that $H\leq K^L$, then we have that $G\leq H\leq K^L$.
  Using part (A) of the induction hypothesis we get that $G$ and $K^L$ have matching parity.
  Then $G$ and $K$ must have mismatching parity, as desired.

  The proof of claim (C) is similar to that of claim (B).
\end{proof}

From this we get the following corollary:

\begin{cor}\label{parity}
  If $G\leqint H$ and both have parity, then they have matching parity. If $G\tri H$ and both have parity, then they have mismatching parity.\end{cor}

This corollary also holds with respect to the contextual order:

\begin{lemma}\label{matching_parity_contextual}
  Suppose $G$ and $H$ are games with parity.
  If $G\contri H$, then $G$ and $H$ have mismatching parity.
  If $G\leqc H$, then $G$ and $H$ have matching parity.  \end{lemma}
\begin{proof}
  Let $G$ and $H$ be defined as above.

  First, suppose $G\contri H$.
  We want to show that $G$ and $H$ cannot have matching parity.
  Suppose that both $G$ and $H$ are even.
  Let $X=\gc{\bot}{\top}$, where $\bot$ and $\top$ are both constant functions which send everything in $A$ to $\bot$ and $\top$, respectively.
  Clearly the first player to play in $X$ loses.
  Then because both games are even, Left always has a second player winning strategy in $G+\gc{\bot}{\top}$ and Left never has a first player winning strategy in the game $H+\gc{\bot}{\top}$, contradicting our assumption.
  For similar reasons using the game $X=\gc{\gc{\bot}{\top}}{\gc{\bot}{\top}}$, we can prove that $G$ and $H$ cannot both be odd.
  This means $G$ and $H$ must have mismatching parity.

  Now, suppose $G\leqc H$.
  When both games are atomic, they immediately have matching parity.

  First, assume $H$ is composite.
  We must have that $o_R(G+X)\leq o_L(H^R+X)$ for all right options of $H$.
  By the proof of the first claim, $G\contri H^R$ implies that $G$ and $H^R$ have mismatching parity. The games $H^R$ and $H$ must have mismatching parity by definition, therefore $G$ and $H$ have matching parity.

  The proof is dual when $G$ is composite and $h$ is atomic.\end{proof}

When turning a $\Rex$ position into a game, the parity of the position is determined by whether there is an even or odd number of empty cells remaining in the position.
Because of this, we have the following corollary:
\begin{cor}
  Let $G$ be a game that is intrinsically equivalent to a position in $\Rex$. Then $G$ must have parity.
\end{cor}

\begin{lemma}\label{bot_G_top}
  Let $G$ be an even game over a poset $A$. Then $\bot\leq G\leq \top$.
\end{lemma}
\begin{proof}
  First, we note that this statement is true by definition when $g$ is atomic.

  We will only prove that $\bot\leq G$, as the proof that $G\leq \top$ is dual.
  Assume $G$ is composite and even.
  Let $G^R$ be a right option of $G$.
  Because $G^R$ is odd, there must be a left option $(G^R)^L$ of $G^R$.
  By the induction hypothesis, $\bot\leq (G^R)^L$.
  Then $\bot\tri G^R$.
  This is true for all right options of $G$.
  The game $\bot$ has no left options.
  Therefore $\bot\leq G$ as desired.
\end{proof}

\section{\texorpdfstring{$*$}{*}-Antimonotonicity}\label{SAM_game}
Recall that $\deadset=\{\,0\,\}$ is the one-element poset.
As discussed in \cref{Background:Posets}, for any poset $A$ we have that ${A\times\deadset} \cateq A$.
Because these are isomorphic, we usually say that $(a,0)=a$.
Let $[0]$ be the unique atomic game over $\deadset$.
Then for all games $G$ we have $G=G+0$.

We also use $*$ to denote the game $\gc{0}{0}$.
Recall the definition of a dead cell from \cref{concrete_*_antimonotonicity}: A cell whose state does not affect the outcome of the game.
Since a black stone is equivalent to a white stone in this cell, the outcome poset should be $\deadset$. 
Then the game $*$ represents a single dead cell.
We have that $*+*\inteq 0$ as well, which can be checked by hand.
\begin{remark}
  If $G$ is even, then $G+*$ is odd, and vice versa.  Also, since $*+*\inteq 0$ we have $G+*+*\inteq G$ by \cref{sum_preserves_order}. 
\end{remark}

By mirroring what we did in \cref{concrete_*_antimonotonicity} we have the following definition:
\begin{definition}[$*$-antimonotonicity]
  A game $G$ is \emph{locally $*$-antimonotone} if \[G^L+*\leq G\leq G^R+*\] for all left and right options of $G$.
  A game is \emph{$*$-antimonotone} if it and all of its followers are locally $*$-antimonotone.
\end{definition}

In \cref{concrete_*_antimonotonicity} it was shown that $\Rex$ positions are $*$-antimonotone, and a counterexample to the claim $G^L\leq G\leq G^R$ was given. In fact, by \cref{parity} we know that $G^L\nleq G\nleq G^R$ for \emph{all} left and right options of $G$ when $G$ has parity. 

\begin{lemma}
  If $G$ is a $*$-antimonotone game, then $G$ has parity.
\end{lemma}
\begin{proof}
  We will prove our claim using induction on $G$.

  All atomic games are even, and hence have parity. Then assume $G$ is a composite $*$-antimonotone game.
  By the induction hypothesis, all options of $G$ have parity. In fact, all $G^L+*$ and all $G^R+*$ have parity. Then by \cref{parity_trans}, the inequality $G^L+*\leq G\leq G^R+*$ implies that, for all $G^L$ and all $G^R$ of $G$, we have $G^L+*$ and $G^R+*$ have matching parity. Then $G^L$ and $G^R$ have matching parity. This means that all options of $G$ have the same parity. Therefore $G$ has parity, as desired.
\end{proof}

\section{Premotivity}

\begin{definition}[Premotivity]
  A game $G$ is \emph{locally premotive} if Left has a first player winning strategy in the comparison game $\lcomp{G}{G}$ (or dually if Right has a first player winning strategy in the game $\rcomp{G}{G}$).
A game is \emph{premotive} if it and all of its followers are locally premotive.\end{definition}

This is a property that all set coloring games have, as seen in \cref{rex_premotivity}.
Premotivity is also preserved by addition, a property that will be useful to us later.

\begin{proposition}\label{sum_premotive}If $G$ and $H$ are premotive, then $G+H$ is premotive.\end{proposition}
\begin{proof}
  We start by assuming $G$ and $H$ are premotive, which implies that Left has a first player winning strategy in the games $\lcomp{G}{G}$ and $\lcomp{H}{H}$.
  By using the copycat strategy, we can show that Left also has a second player winning strategy in both games.
  We want to show that Left has a first player winning strategy in the game $\lcomp{(G+H)}{(G+H)}$.
  By the definition of the Cartesian product of posets, the game $\lcomp{(G+H)}{(G+H)}$ is identical to $\andplus{\lcomp{G}{G}}{\lcomp{H}{H}}$.

  We want to show that Left has a first player winning strategy in this game.
  Their strategy is as follows; Left sets out to play their first player winning strategy in the $\lcomp{G}{G}$ component of the game.
  If Right decides to respond in the $\lcomp{H}{H}$ component of the game, then Left responds with their copycat strategy.
  If Right plays in the $\lcomp{G}{G}$ component of the game Left will follow their second player winning strategy in $(\lcomp{G}{G})^L$.
  Play will continue until Right makes the last move in that section, after which Left will make the first move in some $\lcomp{H'}{H'}$, where $H'$ is either $H$ or some follower of $H$.
  All followers of $H$ are also premotive by definition, and so Left will have a first player winning strategy in the game $\lcomp{H'}{H'}$ as well.
  This fully describes Left's first player winning strategy in the game $\andplus{\lcomp{G}{G}}{\lcomp{H}{H}}$.
\end{proof}

We will now show that premotive games satisfy the lookahead property (see \cref{prop_x}).
\begin{proposition} If $G$ and $H$ are even games over $\bool$ and Left has a first player winning strategy in $\lcomp{G}{H}$, then a second player winning strategy in $G$ implies a first player winning strategy in $H$. \end{proposition}

\begin{proof}
  We will prove the following three claims using joint induction on $G$ and $H$:
  \begin{enumerate}[(A)]
  \item If $G$ and $H$ are even and Left has a first player winning strategy in the game $\lcomp{G}{H}$, then the existence of a second player winning strategy for Left in $G$ implies that there is a first player winning strategy for Left in $H$.
  \item If $G$ is odd, $H$ is even, and Left has a second player winning strategy in $\lcomp{G}{H}$, then the existence of a first player winning strategy for Left in $G$ implies that there is a first player winning strategy for Left in $H$.
  \item If $G$ is even, $H$ is odd, and Left has a second player winning strategy in $\lcomp{G}{H}$, then the existence of a second player winning strategy for Left in $G$ implies that there is a second player winning strategy for Left in $H$.
  \end{enumerate}

  First, we will prove claim (A).
  Assume $G$ and $H$ are even, and that Left has a first player winning strategy in the comparison game $\lcomp{G}{H}$.
  If $g$ and $h$ are atomic, then a first player winning strategy for Left  in $\lcomp{g}{h}$ means that $g\leq h$.
  A first player winning strategy for Left in the atomic game $g$ means that $\top= g$.
  Then, by transitivity of posets, $\top = h$ as well.
  This implies Left has a second player winning strategy in the game $h$, as desired.
  
  Now assume that $G$ and $H$ are not both atomic, and that Left has a second player winning strategy in the game $G$.
  Because Left has a first player winning strategy in the game $\lcomp{G}{H}$, there must be a left option of the game in which Left has a second player winning strategy.
  
  Suppose Left's winning move is of the form $\lcomp{G^R}{H}$ for some left option $\op{G^R}$ of $\op{G}$ (recall that the left options of $\op{G}$ are of the form $\op{G^R}$).
  By assumption Left has a first player winning strategy in this game $G^R$, so by the induction hypothesis Left must have a first player winning strategy in the game $H$.
  
  Suppose instead that Left's winning move is of the form $\lcomp{G}{H^L}$ for some $H^L$ of $H$.
  Left has a second player winning strategy in the game $G$ by assumption, so by the induction hypothesis, Left must have a second player winning strategy in the game $H^L$, which implies they have a first player winning strategy in the game $H$.
  
  Next we will prove claim (B).
  Assume that $G$ is odd, that $H$ is even, that Left has a second player winning strategy in the comparison game $\lcomp{G}{H}$, and that Left has a first player winning strategy in the game $G$.
  This means there exists some $G^L$ of $G$ such that Left has a second player winning strategy in the game $G^L$.
  If Right were to move to $\lcomp{G^L}{H}$, Left must have a first player winning strategy in this game as well.
  Then by the induction hypothesis, Left has a first player winning strategy in $H$.

  The proof of (C) is similar to the proof of (B).  \end{proof}

The lookahead property states that, if some player has a second player winning strategy in an even game $G$, then they have a first player winning strategy in $G$. The previous proposition begets the following corollary:
\begin{cor}\label{o(G)_not_P} Let $H$ be a game over $\bool$. If $H$ is even and premotive and Left has a second player winning strategy in $H$, then there is a first player winning strategy for Left in $H$.\end{cor}

\begin{lemma}\label{cont_eq_premotivity}
  If $G\conteq H$ and $G$ is locally premotive, then $H$ is locally premotive.\end{lemma}

\begin{proof}
  Assume $G\conteq H$ and $G$ is locally premotive.
  Then because $G\conteq H$ (and, dually, because $\op{G}\conteq\op{H}$), we have
  \[o(\lcomp{G}{G})=o(\lcomp{G}{H})=o(\lcomp{H}{H}).\]
  Because $G$ is locally premotive, $o(\lcomp{G}{G})=\gl$, so $o(\lcomp{H}{H})=\gl$.
  Then Left has a first player winning strategy in the game $\lcomp{H}{H}$, meaning $H$ is locally premotive. \end{proof}


Recall that $\emph{and}$ and $\emph{or}$ are the usual boolean functions.
By \cref{sum_on_game} and \cref{map_on_game}, $\plusand$ and $\plusor$ denote two different ways of taking the disjunctive sum of games over $\bool$.
Specifically, to win the game $G\plusand H$, Left must win both components, whereas to win $G\plusor H$, it is sufficient for Left to win one component.

\begin{lemma}\label{sum_and_or}
  Given $G$, $H$, $J$, and $K$ games over a poset $A$, we have that
  $\top\leqint\andplus{\lcomp{G}{H}}{\lcomp{J}{K}}$ implies
  $\top\leqint\orplus{\lcomp{G}{K}}{\rcomp{H}{J}}$. \end{lemma}
\begin{proof}
  Assume $g$, $h$, $j$, and $k$ atoms and assume that $\top\leq\andplus{\lcomp{g}{h}}{\lcomp{j}{k}}$.
  This implies $g\leq h$ and $j\leq k$.
  We want to prove that either $g\leq k$ or $h\not\leq j$.
  
  Suppose $h\leq j$.
  Then $g\leq h\leq j\leq k$ implies $g\leq k$.
  Then we have shown that, if $g\leq h$ and $j\leq k$, then $g\leq k$ or $h\nleq j$.
  This implies that $\top\leqint\orplus{\lcomp{g}{k}}{\rcomp{h}{j}}$.

  The rest of the lemma follows by \cref{fun_equiv_trans}.
\end{proof}


\begin{lemma}\label{trans_lemma_premotive}
  Let $G$ be a game over $\bool$, and let $H$ be a premotive game over any poset.
  Then the following two statements are true:
  \begin{enumerate}[(A)]
  \item If $\top\leqint G\plusor(\rcomp{H}{H})$, then $\top\leqint G$.
  \item If $\top\tri G\plusor(\rcomp{H}{H})$, then $\top\tri G$.
  \end{enumerate} \end{lemma}

\begin{proof}
  Let $G$ and $H$ be defined as above.
  We prove both claims using induction on games $G$ and $H$.
  
  First we prove claim (A).
  Assume that $g$ is atomic.
  We want to show that $\top\leqint g$.
  From \cref{top_implies_win} we see that $\top\leqint g\plusor(\rcomp{H}{H})$ implies Left has a second player winning strategy in the game $g\plusor(\rcomp{H}{H})$.
  Because $H$ is premotive, Right has a first player winning strategy in the game $(\rcomp{H}{H})$, so Left cannot have a second player winning strategy in it.
  Then it must be that the game $g$ is already greater than or equal to $\top$.

  Now assume that $G$ is composite.
  It must be that, for all right options $G^R$ of $G$, we have $\top\tri G^R\plusor(\rcomp{H}{H})$.
  Then by the induction hypothesis, we have that $\top\tri G^R$ for any right option $G^R$ of $G$, which implies $\top\leqint G$.

  Finally we prove claim (B).
  Assume that $\top\tri G\plusor(\rcomp{H}{H})$.
  We want to show that $\top\tri G$.
  Left has some move in either $G$ or $H$ such that $\top\leqint( G\plusor(\rcomp{H}{H}))^L$.
  If that move is of the form $\top\leqint G^L\plusor(\rcomp{H}{H})$, then by part (A) of the induction hypothesis we have that $\top\leqint G^L$, which implies $\top\tri G$.
  If Left's move is of the form $\top\leqint G\plusor(\rcomp{H^L}{H})$, then because $\op{H^L}$ is a right option of $\op{H}$ we must have that $\top\tri G\plusor(\rcomp{H^L}{H^L})$.
  Because $H$ is premotive, all of its options are as well.
  Then by part (B) of the induction hypothesis, we have that $\top\tri G$.
  Similarly, if Left's move is of the form $\top\leqint G\plusor(\rcomp{H}{H^R})$, then we get $\top\tri G\plusor(\rcomp{H^R}{H^R})$ and by part (B) of the induction hypothesis we get that $\top\tri G$.
\end{proof}


Recall from the counterexample in \cref{eq:not-trans} that the intrinsic order is not always transitive.
One of the most important consequences of premotivity is that the intrinsic order is transitive on the class of premotive games.

\begin{theorem} \label{trans_leq} Suppose $G$, $H$, and $K$ are games over the same poset and $H$ is premotive. Then the following three statements hold:
  \begin{enumerate}[(A)]
  \item If $G\leqint H\leqint K$, then $G\leqint K$.
  \item If $G\tri H\leqint K$, then $G\tri K$.
  \item If $G\leqint H\tri K$, then $G\tri K$.
  \end{enumerate}\end{theorem}
\begin{proof}
  First we prove (A).
  Assume $G\leqint H$ and $H\leqint K$.
  From this, we know that $\top\leqint\lcomp{G}{H}$ and $\top\leqint\lcomp{H}{K}$ by \cref{comparison_game_II}.
  The games $\top$, $ \lcomp{G}{H}$, and $\lcomp{H}{K}$ are over the boolean poset, so the function $\andplus{}{}$ can be used to add them together.
  Then by \cref{sum_preserves_order}, we have \[\top\leqint\andplus{\lcomp{G}{H}}{\lcomp{H}{K}}.\]
  Then by \cref{sum_and_or}, we have \[\top\leqint\orplus{\lcomp{G}{K}}{H+_{\rho}H^{op}}.\]
  Because $H$ is premotive, we then have $\top\leqint\lcomp{G}{K}$ by \cref{trans_lemma_premotive}.
  This implies that $G\leq K$ as desired.

  The proofs of (B) and (C) are similar. 
\end{proof}

We finish this section with some properties that relate premotivity and $*$-antimonotonicity.

\begin{lemma}\label{*_cancellation}
  Let $G$ and $H$ be premotive games.
  Then $G+*\leq H$ if and only if $G\leq H+*$. \end{lemma}
\begin{proof}

  First, assume $G+*\leq H$. Then by \cref{sum_preserves_order} we have that $G+*+*\leq H+*$.
  Both $G$ and $*+*$ are premotive, so by \cref{sum_premotive} $G+*+*$ is premotive.
  Because $G+*+*$ is premotive and $G+*+*\inteq G$, we have that $G\leq H+*$ by \cref{trans_leq}.

  The proof of the opposite implication is dual. \end{proof}

\begin{proposition}\label{tri_leq_relation}
  Let $G$ and $H$ be $*$-antimonotone and premotive games over some poset $A$. Then $G\tri H$ if and only if $G\leqint H+*$.\end{proposition}
\begin{proof}
  Because $H$ is a right option of $H+*$, $G\leq H+*$ implies $G\tri H$.
  Because of this, we will only prove that $G\tri H$ implies $G\leq H+*$.

  Assume that $G\tri H$.
  This means either there exists some $H^L$ of $H$ such that $G\leqint H^L$ or there exists some $G^R$ of $G$ such that $G^R\leqint H$.
  Assume there is an $H^L$ as described above.
  Because $H$ is $*$-antimonotone, we have \[G\leqint H^L\leqint H+*.\]
  Because $H$ and all of its followers are premotive, we have transitivity through $H^L$, giving us that $G\leqint H+*$.

  Assume there is a $G^R$ as described above.
  Because $G$ is $*$-antimonotone, we have \[G+*\leqint G^R\leqint H.\]
  Because $G^R$ is premotive, we have that $G+*\leqint H$.
  As shown earlier, this is equivalent to saying $G\leqint H+*$.
\end{proof}

\begin{cor}\label{tri_leq_relation_swap}
  Let $G$ and $H$ be $*$-antimonotone and premotive games.
  Then $G\tri H$ if and only if $G+*\leqint H$.
\end{cor}
\begin{proof}
  This follows directly from \cref{tri_leq_relation} and \cref{*_cancellation}.
\end{proof}

\chapter{The intrinsic and contextual order, united}\label{Intrinsic_extrinsic_united}

We saw in \cref{eq:not-trans} a counterexample to the claim that the intrinsic order was transitive for all games. This already tells us that the contextual order $\leqc$ and the intrinsic order $\leq$ are not equivalent for all games. In this section, we will show that they coincide when restricted to the class of premotive games with parity.

\section{Intrinsic implies contextual}

\begin{lemma}\label{top_implies_win}
  If $\top\leqint G$, then Left has a second player winning strategy in $G$, and if $\top\tri G$, then Left has a first player winning strategy in $G$.
\end{lemma}

\begin{proof}
  We prove both statements by using joint induction on $G$:

  To prove the first statement, let $G$ be a game such that $\top\leq G$.
  If $g$ is atomic, then $\top\leq g$ implies that Left has a first and second player winning strategy in $g$, because the game is already finished.

  Assume $G$ is composite, and let $G^R$ be some right option of $G$.
  Then $\top\tri G^R$.
  By the induction hypothesis, this implies that  Left has a first player winning strategy in $G^R$.
  This argument holds regardless of which $G^R$ we consider.
  Then Left has a first player winning strategy in every right option of $G$, which implies that Left has a second player winning strategy in $G$.

  Now we prove the second statement.
  Let $G$ be a game such that $\top\tri G$.
  Because $\top$ has no right options, there must exist some $G^L$ of $G$ such that $\top\leq G^L$.
  By the induction hypothesis, Left has a second player winning strategy in $G^L$, which implies that Left has a first player winning strategy in $G$.
\end{proof}


\begin{theorem}\label{intrinsic_implies_contextual} For games $G$ and $H$ that are premotive and have parity, if $G\leq H$, then $G\leqc H$. \end{theorem}

\begin{proof}

  To show this, we will prove the following three claims using joint induction on $G$, $H$, and $X$.
  Let $G$ and $H$ be games over some poset $A$ and let $X$ be a game over $\fun{A}$. Assume that $G$ and $H$ are premotive and have parity.
  We will prove the following three statements by joint induction on $G$, $H$, and $X$:
  \begin{enumerate}[(A)]
  \item If $G\leqint H$ and Left has a first player winning strategy in $G+X$, then Left has a first player winning strategy in $H+X$.
  \item If $G\tri H$ and Left has a second player winning strategy in $G+X$, then Left has a first player winning strategy in $H+X$.
  \item If $G\leqint H$ and Left has a second player winning strategy in $G+X$, then Left has a second player winning strategy in $H+X$.
  \end{enumerate}
  Together, claims (A) and (C) imply that $o(G+X)\leq o(H+X)$, which proves $G\leqc H$.
  
  We will begin by proving claim (A).
  Assume that $G\leqint H$ and Left has a first player winning strategy in the game $G+X$.
  If $g$ and $x$ are atomic, then it must be that $\top\leq x(g)$.
  Because $x$ is a monotone function, $g\leqint H$ implies that $x(g)\leqint x(H)$.
  Because all atomic games are premotive, we can use transitivity to say that $\top\leqint x(H)$.
  By \cref{top_implies_win}, $\top\leqint x(H)$ implies Left has a second player winning strategy in $x(H)$.
  Since $g\leq H$ and $H$ has parity, $H$ is even by \cref{parity}.
  Hence $x(H)$ is even as well.
  Because $H$ is premotive, we also have that $x(H)$ is premotive.
  Then by \cref{o(G)_not_P} Left must have a first player winning strategy in $x(H)$, as desired.
  
  Now assume that $G$ and $X$  are not both atomic.
  There must exist some $(G+X)^L$ in which Left has a second player winning strategy.
  First assume this is a move in $X$, resulting in $G+X^L$.
  Then because $G\leqint H$ and because Left has a second player winning strategy in the game $G+X^L$, Left must have a second player winning strategy in the game $H+X^L$ by part (C) of the induction hypothesis.
  Therefore, Left has a first player winning strategy in the game $H+X$.

  Now assume Left's move was not in $X$, but rather to some $G^L$ of $G$.
  Because $G\leqint H$, we know $G^L\tri H$.
  Then because $G^L\tri H$ and because Left has a second player winning strategy in the game $G^L+X$, Left must have a first player winning strategy in the game $H+X$ by part (B) of the induction hypothesis.

  Now we will prove claim (B).
  Assume that $G\tri H$ and Left has a second player winning strategy in $G+X$.

  To say $G\tri H$ is to say that either there exists some $G^R$ such that $G^R\leqint H$ or there exists some $H^L$ such that $G\leqint H^L$.
  First, suppose there exists some $G^R$ such that $G^R\leqint H$.
  Because Left has a second player winning strategy in $G+X$, Left must have a first player winning strategy in the game $G^R+X$.
  This implies Left has a first player winning strategy in the game $H+X$ by part (A) of the induction hypothesis.
  Now suppose there exists some $H^L$ such that $G\leqint H^L$.
  Then Left has a second player winning strategy in the game $H^L+X$ by part (C) of the induction hypothesis, which means Left has a first player winning strategy in the game $H+X$.

  Finally, we prove claim (C).
  Assume that $G\leq H$ and Left has a second player winning strategy in the game $G+X$.
  When $h$ and $x$ are both atomic, the proof is dual to that of claim (A).

  Assume $H$ and $X$ are not both atomic.
  We want to show that Left has a second player winning strategy in the game $H+X$.
  Consider some some right option of $H+X$.
  If this right option is of the form $H+X^R$, then Left must have a first player winning strategy in the game $G+X^R$.
  Then by part (A) of the induction hypothesis, Left has a first player winning strategy in $H+X^R$.
  If this right option is of the form $H^R+X$, then because $G\leq H$ we know $G\tri H^R$.
  Then by part (B) of the induction hypothesis, Left has a first player winning strategy in $H^R+X$.
  Then Left has a second player winning strategy in $H+X$, as desired.
\end{proof}

\section{Contextual implies intrinsic}

\begin{theorem}\label{contextual_implies_intrinsic}
  If $G$ and $H$ have parity, then $G\leqc H$ implies $G\leq H$.\end{theorem}
\begin{proof}

  Assume that $G$ and $H$ have parity and that $G\leqc H$.
  By \cref{matching_parity_contextual}, we know $G$ and $H$ must have matching parity.
  Left certainly has a second player winning strategy in the game $\lcomp{G}{G}$, that being the copycat strategy.
  Because $G\leqc H$, Left also has a second player winning strategy in $\lcomp{G}{H}$.
  This implies $G\leq H$ by \cref{comp_int}.
\end{proof}

\cref{intrinsic_implies_contextual} and \cref{contextual_implies_intrinsic} come together as the following corollary:
\begin{cor}
  Let $G$ and $H$ be games over some poset $A$. If both $G$ and $H$ are premotive and have parity, then $G\leq H$ if and only if $G\leqc H$. 
\end{cor}
\chapter{Canonical forms}\label{Canonical_forms}

In Combinatorial Game Theory, the number of followers a game has heavily impacts how easy it is to analyze a position.
Let $G$ be a game over some poset $A$, and let $X$ be a game over the poset $\fun{A}$.
Suppose $G$ has $n_G$ left options and $m_G$ right options, and $X$ has $n_X$ left options and $m_X$ right options.
Then the game $G+X$ will have $n_G+n_X$ left options and $m_G+m_X$ Right options.

If $G$ is contextually equivalent to an atomic game $h$, then $h+X$ has as many left/right options as $X$ has.
The game $h+X$ will even have the same number of followers as $X$.
Because $G$ and $h$ are contextually equivalent, we have that $G+X\conteq h+X$ and $o(G+X)=o(h+X)$.
This means we can analyze the behavior of $G+X$ while doing only a fraction of the work. 

Now let $\fancy{G}$ be the set of games that are contextually equivalent to $G$.
We can order this set by how many followers each game has.
A \emph{canonical form} of a game $G$ is an element of the set $\fancy{G}$ with the fewest followers.
The games that we are working with only have finitely many followers, so we know that for any game there exists a collection of canonical forms.
In this chapter we will demonstrate how to find these canonical forms and we will prove that, for finite premotive games with parity, there is always a \emph{unique} canonical form. Moreover, if $G$ is also $*$-antimonotone, then its canonical form is as well.

\begin{lemma}\label{options_sam}
  Let $G$ and $H$ be premotive games such that $G\inteq H$.
  If $G$ and all the options of $H$ are $*$-antimonotone, then $H$ is $*$-antimonotone.\end{lemma}
\begin{proof}
  Let $G$ and $H$ be premotive games such that $G\inteq H$, and assume that $G$ and all the options of $H$ are $*$-antimonotone.
  We want to show that $H$ is $*$-antimonotone.
  Let $H^L$ be some left option of $H$.
  Because $H\leqint G$, it must be that $H^L\tri G$.
  By assumption, both $H^L$ and $G$ are $*$-antimonotone and premotive, so by \cref{tri_leq_relation_swap} we have $H^L+*\leq G$.
  Since $G$ is premotive, we have that $H^L+*\leq G\leq H$ implies $H^L+*\leq H$ by transitivity.
  Then we have $H^L+*\leq H$ for all left options $H^L$ of $H$.
  We have $H\leq H^R+*$ for all right options $H^R$ of $H$ for similar reasons.
  Therefore $H$ is $*$-antimonotone. \end{proof}

\begin{lemma}[Gift Horse Lemma]\label{gift_horse}
  Let $G$, $H$, and $G$ be games over some poset $A$ and let $G'=\gc{G^L}{G^R}$ and $G=\gc{H,G^L}{G^R}$.
  Then $G'\leqc G$ and $G'\leq G$. \end{lemma}
\begin{proof}
  The proof of this follows from the definition of $\leq$. \end{proof}

\section{Domination}
One way to simplify a game is to ignore the non-optimal moves. For example, suppose there are two left options $H$ and $K$ such that $H\leq K$. If Left wants to play the move $H$, then Left can play the move $K$ instead. The result will be equivalent or better than if Left had played $H$. We say that the left option $H$ is \emph{dominated}, and we define this formally as follows:
\begin{definition}[Dominated options]
  Let $G$ be a game of the form \\$G=\gc{H,G^L}{G^R}$.
  Then we say that the left option $H$ is \emph{dominated} if $H\leq G^L$ for some $G^L\neq H$ of $G$.
  We say $G'$ is $G$ with the left option $H$ removed if $G'=\gc{G^L}{G^R}$.
  Domination for right options is defined dually.
\end{definition}

In order to make the proof that dominated options can be removed while preserving premotivity, parity, and $*$-antimonotonicity (\cref{dom_preserves}) easier, in the following lemma we assume $H\leqc G^L$ instead of assuming $H\leq G^L$.
We are able to do this because, when $G$ is premotive and has parity, $H\leq G^L$ implies $H\leqc G^L$ by \cref{intrinsic_implies_contextual}.
\begin{lemma}\label{dom_equiv}
  Let $G$ be a game of the form $G=\gc{H,G^L}{G^R}$ such that $H$ is dominated, and let $G'=\gc{G^L}{G^R}$.
  Then $G\conteq G'$. 
\end{lemma}
\begin{proof}
  
  Let $G$ and $G'$ be defined as above, and let $X$ be a game over the poset $\fun{A}$.
  By \cref{gift_horse} we have that $G'\leqc G$, so we only need to prove $G\leqc G'$. 
  We will prove the following claims using induction on $X$:
  \begin{enumerate}[(A)]
  \item If Left has a second player winning strategy in the game $G+X$, then Left has a second player winning strategy in the game $G'+X$.
  \item If Left has a first player winning strategy in the game $G+X$, then Left has a first player winning strategy in the game $G'+X$.
  \end{enumerate}
  
  First we will prove claim (A).
  Assume Left has a second player winning strategy in the game $G+X$.
  We want to show that Left has a second player winning strategy in the game $G'+X$.
  Every right option of $G'$ is a right option of $G$, so those are matched.
  For all $X^R$ of $X$, Left has a first player winning strategy in the game $G+X^R$.
  This implies Left has a first player winning strategy in the game $G'+X^R$ by part (B) of the induction hypothesis.
  Then Left has a first player winning strategy in all the right options of $(G'+X)$, so Left must have a second player winning strategy in $G'+X$. 

  Now we will prove claim (B).
  Assume Left has a first player winning strategy in $G+X$.
  We want to prove that Left has a first player winning strategy in $G'+X$.
  The left options of $G+X$ can be split into three forms: $G+X^L$ for some $X^L$ of $X$, $G^L+X$ for some $G^L\neq H$ of $G$, or $H+X$.
  If Left has a second player winning strategy in $G+X^L$ for some $X^L$ of $X$, then by part (A) of the induction hypothesis we have that Left has a second player winning strategy in $G'+X^L$, which implies Left has a first player winning strategy in $G'+X$.
  Now, assume Left has a second player winning strategy in $G^L+X$ for some $G^L\neq H$ of $G$.
  Because $G^L$ is also a left option of $G'$, Left has a first player winning strategy in $G'+X$.
  Finally, assume Left has a second player winning strategy in $H+X$.
  By assumption, there exists a left option $G^L$ of $G$ such that $H\leqc G^L$.
  Then Left has a second player winning strategy in $G^L+X$ for the $G^L$ described above.
  Then Left has a first player winning strategy in $G'+X$.
\end{proof}

\begin{theorem}[Removing dominated options]\label{dom_preserves}
  Let $G=\gc{H,G^L}{G^R}$ be a premotive game with parity such that $H$ is dominated.
  Let $G'=\gc{G^L}{G^R}$.
  Then $G'$ is premotive and has parity, and $G\inteq G'$.
  Moreover, if $G$ is $*$-antimonotone then $G'$ is $*$-antimonotone.
\end{theorem}
\begin{proof}
  Let $G$ and $G'$ be defined as above.
  Because $G$ is premotive and has parity, $H\leq G^L$ implies that $H\leqc G^L$ by \cref{intrinsic_implies_contextual}.
  Then by \cref{dom_equiv} we have that $G\conteq G'$.
  
  By \cref{cont_eq_premotivity}, $G\conteq G'$ implies that $G'$ is locally premotive.
  Because $G$ is premotive and every follower of $G'$ is a follower of $G$, we have that every follower of $G'$ is premotive.
  Then $G'$ is premotive.
  Because $G$ and $G'$ both have parity and are premotive, we have $G\inteq G'$ by \cref{contextual_implies_intrinsic}.

  Suppose $G$ is $*$-antimonotone.
  All of the options of $G'$ are $*$-antimonotone, so by \cref{options_sam} we have that $G'$ is $*$-antimonotone. 
\end{proof}

\section{Reversibility}
Another method for simplifying games that is commonly used in combinatorial game theory is bypassing reversible options. Reversible options are defined as follows:

\begin{definition}[Reversible options]
  Let $G=\gc{H,G^L}{G^R}$ be a game such that $K\leq G$ for some right option $K$ of $H$.
  Then we say that $H$ is a \emph{reversible left option} of $G$, and we say that $H$ is \emph{reversible through} $K$. 
  We define a \emph{reversible right option} of a game $G$ dually.
\end{definition}

\begin{definition}[Bypassing a reversible option]
  Let $G$ be a game with a left option $H$ that is reversible through a right option $K$.
  When we \emph{bypass the reversible option} from the game $G$ it becomes the game $G'\identical\gc{K^L,G^L}{G^R}$. Here, $K^L$ denotes all left options of $K$ when $K$ is composite, and when $k$ is atomic $K^L=\gc{\bot}{k}$.

  Bypassing a reversible right option is defined dually. Note that, when bypassing a reversible right option, $K^R=\gc{k}{\top}$ when $k$ is atomic.
\end{definition}

Note: When $k$ is atomic, then $k$ is contextually equivalent to the composite game $\atomize{k}$.
This is why in Definition 8.7 we use $\gc{\bot}{k}$ as a stand in for a left option of $k$ when $k$ is atomic.

\begin{remark}
  If $G$ has a reversible option and parity then $G'$, the game $G$ with the reversible option bypassed, also has parity. \end{remark}


In order to make the proof that reversible options can be bypassed while preserving premotivity, parity, and $*$-antimonotonicity (\cref{rev_preserves}) easier, in the following lemma we assume $K\leqc G$ instead of assuming $K\leq G$.
We are able to do this because, when $G$ and $K$ are premotive and have parity, $K\leq G$ implies $K\leqc G$.

\begin{lemma}
  Let $G$ be a game with a left option $H$, and a right option $K$ of $H$ such that $K\leqc G$.
  Then $G\conteq G'$, where $G'=\gc{K^L,G^L}{G^R}$.\label{rev_equiv}\end{lemma}

\begin{proof}
  Let $X$ be a game over the poset $\fun{A}$, let $G$ be a game of the form $\gc{H,G^L}{G^R}$, where $K$ is a right option of $H$ such that $K\leqc G$, and let $G'\identical\gc{K^L,G^L}{G^R}$.
  We will prove the following three claims using induction on the depth of game $X$:
  \begin{enumerate}[(A)]
  \item Left has a second player winning strategy in the game $G+X$ if and only if Left has a second player winning strategy in the game $G'+X$.
  \item Left has a first player winning strategy in the game $G+X$ if and only if Left has a first player winning strategy in the game $G'+X$.
  \end{enumerate}

  First we will prove claim (A).
  We will only prove the left-to-right implication, as the proof of the right-to-left implication is nearly identical.
  
  Assume that Left has a second player winning strategy in the game $G+X$.
  We want to show that Left has a second player winning strategy in the game $G'+X$.
  Because Left has a second player winning strategy in the game $G+X$, Left must have a first player winning strategy in all right options of the game.
  We know that Left has a first player winning strategy in the game $G^R+X$ - where $G^R$ is a right option of $G'$ - because the every right option of $G'$ is also a right option of $G$.
  Now let $G'+X^R$ be some right option of the game $G'+X$.
  Because Left must have a first player winning strategy in the game $G+X^R$, Left must also have a first player winning strategy in the game $G'+X^R$ by part (C) of the induction hypothesis.
  This means that Left has a first player winning strategy in every right option of the game $G'+X$, which implies that Left has a second player winning strategy in the game $G'+X$ as desired.

  Next we will prove claim (B).
  The proof of the left-to-right implication is as follows:
  
  Assume that Left has a first player winning strategy in the game $G'+X$.
  We want to show that Left has a first player winning strategy in the game $G+X$.
  By assumption, Left must have a second player winning strategy in some left option of $G'+X$.
  This left option could come in one of three forms: the form $G+X^L$ for some $X^L$ of $X$, the form $G^L+X$ for some $G^L$ of $G$ other than $H$, or the form $K^L+X$ for some $K^L$ of $K$, where $K$ is specified above.

  If Left has a second player winning strategy in a game of the first form, $G+X^L$, then Left must have a second player winning strategy in the game $G'+X^L$ by part (A) of the induction hypothesis, which means that Left has a first player winning strategy in the game $G+X$.
  If Left has a second player winning strategy in a game of the second form, $G^L+X$, then Left must have a first player winning strategy in the game $G+X$.
  Now assume that Left has a second player winning strategy in a game of the third form, $K^L+X$.
  This means that Left has a first player winning strategy in the game $K+X$.
  Because $K\leqc G$, Left must then have a first player winning strategy in the game $G+X$.
  Then Left having a first player winning strategy in $G'+X$ implies the existence of a first player winning strategy for Left in $G+X$.
  
  Finally, we will prove the right-to-left implication.
  Assume that Left has a first player winning strategy in the game $G+X$.
  We want to show that Left has a first player winning strategy in the game $G'+X$.
  By assumption Left must have a second player winning strategy in some left option of $G+X$.
  If the left option in question is of the form $G+X^L$ or $G^L+X$ (for $G^L\neq H$), then the same reasoning used in the ($\Leftarrow$) part of (B) applies.
  Assume that Left has a second player winning strategy in the game $H+X$.
  Then Left has a first player winning strategy in the game $K+X$.

  Because $K\leqc G$, if there exists some $X^L$ such that Left has a second player winning strategy in $K+X^L$, then Left must have a second player winning strategy in $G+X^L$, and so we repeat the argument for the case above.
  Then suppose there is some $K^L$ of $K$ such that Left has a second player winning strategy in $K^L+X$.
  Because $K^L$ is a left option of $G'$, Left has a first player winning strategy in the game $G'+X$.
\end{proof}

\begin{theorem}\label{rev_preserves}\label{rev_premotive}
  Let $G$ be a premotive game with parity, and suppose G has a left option $H$ that is reversible through $K\leq G$, and let $G'$ be the game $G$ with its reversible option bypassed.
  Then $G'$ is premotive and has parity and $G\inteq G'$.
  Moreover, if $G$ is $*$-antimonotone then $G'$ is as well. \end{theorem}
\begin{proof}

  Assume $G$ and $G'$ are as above.
  Because $G$ and $K$ are both premotive and have parity, we have $K \leqc G$ by \cref{intrinsic_implies_contextual}.
  By \cref{rev_equiv}, we know $G\conteq G'$.
  By \cref{cont_eq_premotivity} we have that $G'$ is also premotive.
  The left option $H$ of $G$ has the same parity as each $K^L$, so the game $G'$ has parity as well.
  Both $G$ and $G'$ are premotive and have parity and $G\conteq G'$, which implies $G\inteq G'$ by \cref{contextual_implies_intrinsic}.

  Now, suppose that $G$ is $*$-antimonotone. Every option of $G'$ is a follower of $G$, so all of the options of $G'$ are $*$-antimonotone.
  Therefore $G'$ is $*$-antimonotone by \cref{options_sam}. 
\end{proof}


\begin{remark}[A note on $\atomize{k}$]\label{remark:simple}
  Both the left and right options of the game $\atomize{k}$ are reversible through the atom $k$. When we try to bypass these reversible options, however, we end up returning to $\atomize{k}$.
  This becomes a problem when we are trying to bypass all the reversible options of a game.
  We generalize this issue in the following definition:
  \begin{definition}
    Let $G$ be a game over some poset $A$ with a reversible left option $H$.
    We say $H$ is a \emph{simple} reversible left option if $H=\gc{\bot}{k}$ for some $k\in A$.
    Simple reversible right options are defined dually.
  \end{definition}

\end{remark}

Next are two lemmas relating to these simple reversible options that we will use in the following section.
\begin{lemma}
  Let $x$ and $G$ be games over the poset $A$ such that $x$ is atomic and $G$ is composite.
  If $G$ is odd and Left has a second player winning strategy in $\lcomp{x}{G}$, then $\gc{\bot}{x}\leq G$.
\end{lemma}
\begin{proof}
  We will prove this claim via induction on the game $G$.
  Let $G$ and $x$ be games as described above, and assume Left has a second player winning strategy in the game $\lcomp{x}{G}$.
  We want to show that $\gc{\bot}{x}\leq G$.
  First, we know that $\bot\tri G$ by \cref{bot_G_top}.
  Next we must show that $\gc{\bot}{x}\tri G^R$ for all right options $G^R$ of $G$.
  Let $G^R$ be a right option of $G$.
  By assumption, Left must has a first player winning strategy in the game $\lcomp{x}{G^R}$.
  If $G^R=a$ is atomic, then our strategy tells us that $x\leq a$.
  Assume $G^R$ is composite.
  Then let $G^{RL}$ be a left option of $G^R$ such that Left has a second player winning strategy in the game $\lcomp{x}{G^{RL}}$.
  Then, by the induction hypothesis, we have $\gc{\bot}{x}\leq G^{RL}$.
  So $\gc{\bot}{x}\tri G^R$ for all right options $G^R$ of $G$, which implies $\gc{\bot}{x}\leq G$, as desired. 
\end{proof}

\begin{cor}\label{G_to_GL}
  Let $x$ and $G$ be premotive games over some poset $A$ such that $x$ is atomic and $G$ is composite and has parity.
  If $x\leq G$, then there is a left option $G^L$ of $G$ such that $\gc{\bot}{x}\leq G^L$.
\end{cor}
\begin{proof}
  Let $x$ and $G$ be games as described above, and assume $x\leq G$.
  By \cref{comp_int} we know Left has a second player winning strategy in the game $\lcomp{x}{G}$.
  By \cref{o(G)_not_P} we know Left has a first player winning strategy in the game $\lcomp{x}{G}$.
  Then, because $G$ is composite, there must exist some $G^L$ of $G$ such that Left has a second player winning strategy in the game $\lcomp{x}{G^L}$.
  Because $G$ is even, the game $G^L$ is odd.
  Then by the previous lemma we have that $\gc{\bot}{x}\leq G^L$. 
\end{proof}

\section{Canonical Forms}

When studying normal play games, we find the canonical form of a game $G$ by removing all of its dominated options, bypassing all of its reversible options, and recursively repeating this process with the left and right options of $G$.
We would like to do something similar in $\Rex$, but we run into problems when bypassing reversible options.
If $G$ is a game with a simple reversible option, and $G'$ is $G$ with that reversible option bypassed, then by \cref{remark:simple} we have that $G'$ is identical to $G$.
This means that there are games in which it is impossible to bypass all of its reversible options: for example, take the game
\[G\identical\gc{\gc{\bot}{a},\,\gc{\bot}{b}}{\gc{a,b}{\top}}.\]
We have that $G$ is $*$-antimonotone, premotive, and has parity, and we also have that $a\leq G$ and $b\leq G$.
Then both of the left options of $G$ are simple reversible options.

In order to treat this problem, we may attempt to say a game $G$ is in canonical form if none of its options are dominated, all of its reversible options are simple, and the left and right options of $G$ are in canonical form.
This definition of canonical form escapes the looping problem and finds the smallest game equivalent to $G$, with the exception of $\atomize{x}$, which cannot be further reduced even though it is equivalent to $x$.
Because of this, in the case where $G$ is equivalent to an atom, the process described above may not result in the atomic game expected.
As the next proposition shows, this is the \emph{only} exception.

\begin{proposition}\label{canon_atomize}
  Let $G$ be a composite premotive game with parity over some poset $A$ whose options are not dominated and whose reversible options are simple.
  If $G\inteq x$ for some $x$ in $A$, then $G\identical \atomize{x}$.
\end{proposition}
\begin{proof}
  Let $G$ be a composite game as described above, and assume $G\inteq x$ for some $x$ in $A$.
  We want to show that $G$ is identical to $\atomize{x}$.

  Consider any $G^L$ of $G$.
  From $G\leq x$, we have $G^L\tri x$.
  This implies that there is some $G^{LR}$ such that $G^{LR}\leq x$.
  Because $x\inteq G$, we have that $G^L$ is a reversible left option.
  By assumption $G^L$ is simple, i.e., $G^L=\gc{\bot}{y}$ for some $y\in A$, and we have that $y=G^{LR}\leq x$.
  Every left option of $G$ is of the form $\gc{\bot}{y}$ with $y\leq x$.
  Also, because $x\leq G$, there must exist some ${G^L}'$ such that $\gc{\bot}{x}\leq {G^L}'$ by \cref{G_to_GL}.
  The game ${G^L}'$ dominates every left option of $G$.
  We assumed that $G$ had no dominated options, so ${G^L}'$ must be the only left option.
  We know ${G^L}'=\gc{\bot}{y}$ and we know $\gc{\bot}{x}\inteq {G^L}'$.
  Therefore, the only left option of $G$ is ${G^L}'=\gc{\bot}{x}$.
  Dually, the only right option of $G$ is $\gc{x}{\top}$, proving the claim. \end{proof}

Because of \cref{canon_atomize}, we define a canonical game to exclude the case where $G\identical\atomize{k}$.

\begin{definition}
  A game $G$ over the poset $A$ is \emph{canonical} if all of the following hold:
  \begin{itemize}
  \item None of the options of $G$ are dominated,
  \item the only reversible options of $G$ are simple,
  \item $G$ is not of the form $\atomize{x}$ for some atom $x$ in $A$, and
  \item all left/right options of $G$ are canonical.
  \end{itemize}
\end{definition}


\begin{theorem}[Uniqueness of canonical forms]
  Let $G$ and $H$ be premotive games with parity over the poset $A$ which are canonical.
  If $G\inteq H$, then $G\identical H$.
\end{theorem}
\begin{proof}
  Let $G$ and $H$ be as above.
  We want to show that $G\identical H$.
  If $g$ and $h$ are atomic, then because $g\inteq h$ we know that $g=h$ by definition.
  Assume that $G$ or $H$ is composite.
  By \cref{canon_atomize} and the definition of a canonical game, we have that both $G$ and $H$ must be composite.
  First we will prove the following four claims:
  \begin{enumerate}[(A)]
  \item For every $G^L$ of $G$, there exists some $H^L$ of $H$ such that $G^L\leq H^L$.
  \item For every $H^L$ of $H$, there exists some $G^L$ of $G$ such that $H^L\leq G^L$.
  \item For every $G^R$ of $G$, there exists some $H^R$ of $H$ such that $G^R\leq H^R$.
  \item For every $H^R$ of $H$, there exists some $G^R$ of $G$ such that $H^R\leq G^R$.
  \end{enumerate}
  
  We will prove claim (A).
  Let $G^L$ be a left option of $G$.
  Because $G\leq H$, we know that $G^L\tri H$.
  From this, we know that either there exists some $H^L$ of $H$ such that $G^L\leq H^L$ or there exists some $G^{LR}$ of $G^L$ such that $G^{LR}\leq H$.
  If we have $H^L$ such that $G^L\leq H^L$, then we are done.
  Assume we have a $G^{LR}\leq H$ for some $G^{LR}$ of $G^L$.
  Because $H\leq G$ and $H$ is premotive, by transitivity we have that $G^{LR}\leq G$, so $G^L$ is a reversible left option of $G$.
  Since $G$ is canonical, $G^L$ is simple, i.e., $G^L\identical \gc{\bot}{a}$ where $a$ is an atomic game.
  Then because $a\leq H$, by \cref{G_to_GL} we know there exists a left option $H^L$ of $H$ such that $\gc{\bot}{a}\leq H^L$, so $G^L\leq H^L$ as desired.
  The proof of (B), (C), and (D) are similar.

  Let $G^\gl$ be the set of left options of $G$, and similarly for $G^\gr$, $H^\gl$, and $H^\gr$.
  We will now prove that $G^\gl=H^\gl$.
  Let $G^L$ be some left option of $G$.
  Then by (A) there exists some $H^L$ of $H$ such that $G^L\leq H^L$.
  By (B) there exists some ${G^L}'$ of $G$ such that $H^L\leq {G^L}'$.
  Because $H^L$ is premotive, by transitivity we have that $G^L\leq {G^L}'$.
  Because $G$ has no dominated options, it must be that $G^L\identical {G^L}'$.
  Then $G^L\leq H^L\leq G^L$, which implies $G^L\inteq H^L$.
  The games $G^L$ and $H^L$ are canonical, so by the induction hypothesis $G^L\identical H^L$.
  Then $G^\gl$ is a subset of $H^\gl$.
  By similar reasoning we have that $H^\gl$ is a subset of $G^\gl$.
  Thus, we have that $G^\gl=H^\gl$.
  The proof that $G^\gr=H^\gr$ is dual.
  Therefore, the game $G$ is identical to the game $H$, as desired.\end{proof}

We say that a game is \emph{finite} if it has finitely many followers.

If $G$ is a finite game, then we can repeatedly remove all of its dominated options, bypass all of its non-simple reversible options, and replace \\$\atomize{x}$ by $x$.
At the end of this process we will have a canonical form.

Therefore we have the following:
\begin{cor}[Existence of canonical forms]
  Every finite game has a unique canonical form.
\end{cor}

\chapter{Conclusion}
\section{Analysis of large games of \texorpdfstring{$\Rex$}{Rex}}
Now that we have a method for finding the unique canonical form of a game, we can use this power to more easily analyze large games of $\Rex$.
We will now demonstrate how this may be done by determining the outcome class of the game from \cref{Background}. Here is the original position:

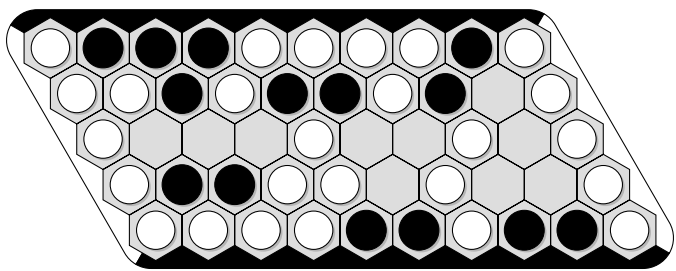
\[\begin{hexboard}
    \rotation{-30}
    \board(10,5)
    \white(1,1)\black(2,1)\black(3,1)\black(4,1)\white(5,1)\white(6,1)\white(7,1)\white(8,1)\black(9,1)\white(10,1)
    \white(1,2)\white(2,2)\black(3,2)\white(4,2)\black(5,2)\black(6,2)\white(7,2)\black(8,2)\white(10,2)
    \white(1,3)\white(5,3)\white(8,3)\white(10,3)
    \white(1,4)\black(2,4)\black(3,4)\white(4,4)\white(5,4)\white(7,4)\white(10,4)
    \white(1,5)\white(2,5)\white(3,5)\white(4,5)\black(5,5)\black(6,5)\white(7,5)\black(8,5)\black(9,5)\white(10,5)
  \end{hexboard}\]

And here is our suggested way of splitting the board:

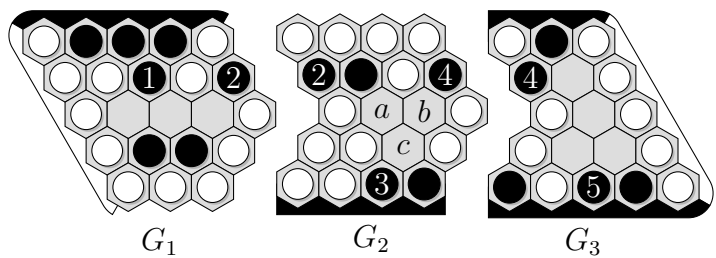
\[
    \begin{hexboard}[scale=0.80, baseline=(current bounding box.center)]
      \rotation{-30}
      \foreach \i in {1,...,5} { \hex(\i,1) }
      \white(1,1)\black(2,1)\black(3,1)\black(4,1)\white(5,1)
      \foreach \i in {1,...,5} { \hex(\i,2) }
      \white(1,2)\white(2,2)\black(3,2)\label{$1$}\white(4,2)\black(5,2)\label{$2$}
      \foreach \i in {1,...,5} { \hex(\i,3) }
      \white(1,3)\white(5,3)
      \foreach \i in {1,...,4} { \hex(\i,4) }
      \white(1,4)\black(2,4)\black(3,4)\white(4,4)
      \foreach \i in {1,...,3} { \hex(\i,5) }
      \white(1,5)\white(2,5)\white(3,5)
      \edge[\nw\noobtusecorner] (1,1)(5,1)
      \edge[\sw] (1,1)(1,5)
      \node at \coord(1,6.5) {$G_1$};
    \end{hexboard} 
    \begin{hexboard}[scale=0.80, baseline=(current bounding box.center)]
      \rotation{-30}
      \foreach \i in {5,...,8}{ \hex(\i,1) }
      \white(5,1)\white(6,1)\white(7,1)\white(8,1)
      \foreach \i in {5,...,8}{ \hex(\i,2) }
      \black(5,2)\label{$2$}\black(6,2)\white(7,2)\black(8,2)\label{$4$}

      \hex(5,3)\white(5,3) \hex(6,3)\label{$a$} \hex(7,3)\label{$b$} \hex(8,3)\white(8,3)
      \foreach \i in {4,5,7}{ \hex(\i,4) }
      \white(4,4)\white(5,4)\hex(6,4)\label{$c$}\white(7,4)
      \foreach \i in {3,...,6}{ \hex(\i,5) }
      \white(3,5)\white(4,5)\black(5,5)\label{$3$}\black(6,5)
      \edge[\se\noobtusecorner\noacutecorner](3,5)(6,5)
      \node at \coord(4,6.5) {$G_2$};
    \end{hexboard} 
    \begin{hexboard}[scale=0.80, baseline=(current bounding box.center)]
      \rotation{-30}
      \foreach \i in {8,...,10}{ \hex(\i,1) }
      \white(8,1)\black(9,1)\white(10,1)
      \foreach \i in {8,...,10}{ \hex(\i,2) }
      \black(8,2)\label{$4$}\white(10,2)
      \foreach \i in {8,...,10}{ \hex(\i,3) }
      \white(8,3)\white(10,3)
      \foreach \i in {7,...,10}{ \hex(\i,4) }
      \white(7,4)\white(10,4)
      \foreach \i in {6,...,10}{ \hex(\i,5) }
      \black(6,5)\white(7,5)\black(8,5)\label{$5$}\black(9,5)\white(10,5)
      \edge[\nw\noacutecorner](8,1)(10,1)
      \edge[\ne](10,1)(10,5)
      \edge[\se\noobtusecorner](6,5)(10,5)
      \node at \coord(7,6.5) {$G_3$};
    \end{hexboard}
  \]

  First we will analyze $G_1$.
  We can view this as a two-terminal position, with \blackstone*{$1$} and \blackstone*{$2$} being the north and south terminals, respectively.
  Then the outcome poset is $\bool$.
  For ease of addition, however, we will refer to the atoms by which terminals are connected:
  One can show that the canonical form of $G_1$ is $\gc{(1,2)}{\top}$.

  The game $G_3$ is also a position with two terminals, labeled \blackstone*{$4$} and \blackstone*{$5$}.
  Black has a second player winning strategy in $G_3$ and the game is even, so its canonical form is $\top$.

  As shown in \cref{fig:poset_G2}, the game $G_2$ is a 3-terminal position. The canonical form of $G_2$ is quite large, so it is helpful to break it into subpositions.
  If Left were to play in the cell \inlinehex*{$a$}, one can show that the resulting game would have the following canonical form:
  \[C_a=\gc{\gc{(2,3,4)}{(2,3)},\,\gc{(2,3,4)}{(2,4)}}{\gc{(2,3)}{\top},\,\gc{(2,4)}{\top}}.\]
  Similarly, if Left were to play in the cell \inlinehex*{$b$} or \inlinehex*{$c$}:
  \[C_b=\gc{\gc{(2,3,4)}{(2,4)},\,\gc{(2,3,4)}{(3,4)}}{\gc{(2,4)}{\top},\,\gc{(3,4)}{\top}}.\]
  \[C_c=\gc{\gc{(2,3,4)}{(2,3)},\,\gc{(2,3,4)}{(3,4)}}{\gc{(2,3)}{\top},\,\gc{(3,4)}{\top}}.\]
  
  It turns out that the canonical form of $G_2$ is $\gc{C_a,\,C_b,\,C_c}{\top}$.

  Now that we have found the canonical form for each region, all we need is a map $f\colon\bool\times P_3\times \bool\to\bool$ such that $f(a,b,c)=\bot$ if and only if the northern terminal of the larger board is connected to the southern terminal, i.e.,  when \blackstone*{$1$} or \blackstone*{$4$} is connected to \blackstone*{$3$} or \blackstone*{$5$}.
  
  The canonical form of $f(G_1+G_2+G_3)$ is $\top$, meaning that $o(G)=\gl$.
  You may be wondering why the canonical form of $G_2$ is so much larger than that of $G_1$, $G_3$, and $f(G_1,G_2,G_3)$.
  As it turns out, there are only $6$ games over $\bool$ up to equivalence that are both premotive and $*$-antimonotone.
  Of course, we must have $\top$ and $\bot$, along with $X=\gc{\bot}{\top}$.
  The other $3$ games are $\top+*$, $\bot+*$, and $X+*$. This means that there are only $6$ possible options for the canonical form of a game over $\bool$, all of which are rather small.
  
  \section{Future work}
  
  \emph{Realizability of games as positions in $\Rex$:}
  While we can always turn a $\Rex$ position into a game form, not every game form can be realized as a $\Rex$ position.
  For example, if a game form is not $*$-antimonotone and/or premotive then it cannot be realized as a position in $\Rex$.
  Some research has been done into what positions are $\Hex$ realizable,
  and one could use the work in this thesis to explore similar ideas for $\Rex$.

  \emph{PSPACE completeness of $\Rex$:}
  Many combinatorial games have been shown to be PSPACE complete.
  In fact, some antimonotone set coloring games have been shown to be PSPACE complete \cite{Avoidance_pspace}.
  We would like to use our method of determining canonical forms to develop gadgets,
  which we will use to determine if $\Rex$ is PSPACE complete or not.

\bibliographystyle{plain} 
\bibliography{citations} 
\end{document}